\documentclass{article}
\usepackage[T1]{fontenc}
\usepackage{amsthm, fullpage}
\usepackage{amsmath}
\usepackage{amssymb}
\usepackage{amsfonts}
\usepackage{eucal}
\usepackage[all]{xy}
\usepackage[frenchb, english]{babel}
\usepackage{pslatex}
\usepackage[pagebackref]{hyperref}

\newtheorem{prop}{Proposition}[section]

\newtheorem*{theo**}{Théorème}

\newtheorem*{conj*}{Conjecture}
\newtheorem{lemm}[prop]{Lemme}
\newtheorem{lemm*}{Lemme}[prop]

\theoremstyle{definition}

\newtheorem{vide}[prop]{}
\newtheorem{defi}[prop]{Définition}
\newtheorem*{defi*}{Définition}

\theoremstyle{remark}
\newtheorem{rema}[prop]{Remarques}

\newtheorem{nota}[prop]{Notations}

\numberwithin{equation}{prop}

\newcommand{\riso}{ \overset{\sim}{\longrightarrow}\, }
\newcommand{\liso}{ \overset{\sim}{\longleftarrow}\, }

\renewcommand{\sp}{\mathrm{sp}}

\newcommand{\FF}{{\mathcal{F}}}

\newcommand{\E}{{\mathcal{E}}}

\newcommand{\D}{{\mathcal{D}}}

\newcommand{\PP}{{\mathcal{P}}}
\newcommand{\QQ}{{\mathcal{Q}}}
\renewcommand{\O}{{\mathcal{O}}}

\newcommand{\V}{\mathcal{V}}

\newcommand{\A}{\mathbb{A}}

\newcommand{\DD}{\mathbb{D}}
\renewcommand{\L}{\mathbb{L}}
\newcommand{\R}{\mathbb{R}}
\newcommand{\Q}{\mathbb{Q}}

\newcommand{\hdag}{  \phantom{}{^{\dag} }    }

\begin{document}
\selectlanguage{frenchb}

\title{Le formalisme des six opérations de Grothendieck en cohomologie $p$-adique}
\author{Daniel Caro} 

\date{}

\maketitle

\selectlanguage{english}
\begin{abstract}
Let $\mathcal{V}$ be a complete discrete valued ring of mixed characteristic $(0,p)$, 
$K$ its field of fractions,  
$k$ its residue field which is supposed to be perfect. 
Let $X$ be a separated $k$-scheme of finite type and
$Y$ be an open subscheme of $X$. 
We construct the category
$F\text{-}D ^\mathrm{b} _\mathrm{ovhol}  (\mathcal{D} ^\dag _{(Y,X)/K})$
 of overholonomy type over $(Y,X)/K$.
We check that these categories satisfy a formalism of Grothendieck's six operations.
\end{abstract}

\selectlanguage{frenchb}
\date
\tableofcontents

\bigskip 

\section*{Introduction}
Soit $\V$ un anneau de valuation discrète complet d'inégales caractéristiques $(0,p)$, 
de corps résiduel parfait $k$, de corps des fractions $K$. 
Soit $(Y, X)$ un couple, i.e. soit $Y \subset X$ une immersion ouverte de $k$-variétés (i.e. de $k$-schémas de type fini).
S'il existe un cadre  de la forme
$(Y, X, \PP)$ , i.e. s'il existe 
$\PP$ un $\V$-schéma formel (pour la topologie $p$-adique) séparé lisse
et $X \hookrightarrow \PP$
une immersion fermée, 
on note alors $F\text{-}D ^\mathrm{b} _\mathrm{surhol}  (Y, X, \PP/K)$  
la sous-catégorie pleine de la catégorie dérivée des $F\text{-}$complexes 
surholonomes de $\D ^\dag _{\PP,\Q}$-modules (toujours à gauche par défaut) à cohomologie bornée
dont les objets sont les $F\text{-}$complexes $\E$ 
tels qu'il existe un isomorphisme de la forme 
$\R \underline{\Gamma} ^\dag _{Y} (\E)\riso \E$.
On a déjà vérifié que l'on dispose d'un formalisme des six opérations de Grothendieck sur ce type de catégories
(voir notamment \cite{caro_surholonome} et \cite{caro-Tsuzuki}).
Nous expliquons dans ce papier comment étendre la construction de ce genre de catégories 
sans l'hypothèse de l'existence d'un cadre  englobant le couple $(Y, X)$.
La catégorie ainsi construite sera notée
$F\text{-}D ^\mathrm{b} _\mathrm{surhol}  (\D ^\dag _{(Y,X)/K})$.
Nous étendons ensuite naturellement le formalisme des six opérations  de Grothendieck sur les catégories 
de la forme $F\text{-}D ^\mathrm{b} _\mathrm{surhol}  (\D ^\dag _{(Y,X)/K})$.

Précisons à présent le contenu de ce papier. 
Dans le premier chapitre, nous donnons quelques rappels, définitions et propriétés sur les 
catégories définies sur les cadres. Dans le second chapitre, 
en nous inspirant du procédé utilisé en cohomologie rigide (voir \cite{LeStum-livreRigCoh}),
nous construisons la catégorie 
notée 
$F\text{-}D ^\mathrm{b} _\mathrm{surcoh}  (\D ^\dag _{(Y,X)/K})$
des complexes de type surcohérent sur le couple $(Y,X)$. 
Puis, dans le chapitre qui suit, on construit la catégorie des modules de type surcohérent et celle
des isocristaux surconvergents. 
Dans le quatrième chapitre, 
nous définissons naturellement l'image inverse extraordinaire, 
et le produit tensoriel. 
La construction de l'image directe est plus délicate: on parvient à la définir pour un morphisme réalisable de couples  (e.g. si
le morphisme $X' \to X$ est propre et le morphisme induit $Y' \to Y$ est quasi-projectif 
ou si le morphisme de couples se prolonge en un morphisme de cadres). 
Dans le cinquième chapitre, 
nous nous intéressons ensuite à l'indépendance par rapport à $X$ de la catégorie
$F\text{-}D ^\mathrm{b} _\mathrm{surcoh}  (\D ^\dag _{(Y,X)/K})$. 
Il manque cependant la stabilité par foncteur dual. 
Cela nous amène à introduire dans le sixième chapitre une catégorie de type dual surcohérent 
notée $F\text{-}D ^\mathrm{b} _\mathrm{surcoh}  (\D ^\dag _{(Y,X)/K})^{*}$ qui vérifie les propriétés duales 
de celles des complexes de type surcohérent: 
on dispose sur  les catégories de la forme $F\text{-}D ^\mathrm{b} _\mathrm{surcoh}  (\D ^\dag _{(Y,X)/K})^{*}$
des opérations images inverses, images directes extraordinaires par un morphisme réalisable
ainsi que du produit tensoriel tordu (on pourrait aussi dire "dualisé").
On bénéficie de plus de l'isomorphisme quasiment tautologique 
induit par le foncteur dual
$\DD _{(Y,X)/K}
\colon 
F\text{-}D ^\mathrm{b} _\mathrm{surcoh}  (\D ^\dag _{(Y,X)/K}) 
\riso
F\text{-}D ^\mathrm{b} _\mathrm{surcoh}  (\D ^\dag _{(Y,X)/K}) ^{*}$.
Pour obtenir la stabilité par les six opérations, 
on introduit dans le septième chapitre la catégorie 
notée 
$F\text{-}D ^\mathrm{b} _\mathrm{surhol}  (\D ^\dag _{(Y,X)/K})$
des complexes de type surholonome sur $(Y,X)$, qui est une sorte de recollement des 
catégories de complexes de type surcohérent et de type dual surcohérent.
Enfin, dans le dernier chapitre, nous adaptons toutes les constructions précédentes
pour les $k$-variétés au lieu des couples de $k$-variétés. 
\bigskip

{\bf Convention:} 
Les $\V$-schémas formels seront notés par des lettres calligraphiques, 
leur fibre spéciale par la lettre droite associée. 
Les $k$-schémas sont toujours réduits.

\section{Opérations cohomologiques sur les cadres}
\begin{defi}
On définit respectivement la catégorie des {\og $d$-cadres\fg}  
et celle des {\og $d$-cadres localement propre \fg}
de la manière suivante.
\begin{enumerate}
\item 
Un $d$-cadre est la donnée 
d'un $\V$-schéma formel séparé et lisse $\PP$, 
d'une immersion ouverte de $k$-schémas 
$Y \subset X$,
d'une immersion fermée 
$X \hookrightarrow \PP$,
d'un diviseur $T$  de $P$ tel que $Y = X \setminus T$.
On note  $(\PP, T ,X, Y)$ un tel $d$-cadre. 

Un morphisme de $d$-cadres
$(f, a,b) \colon (\PP', T', X', Y') \to (\PP, T,  X, Y)$ 
est un morphisme de $\V$-schémas formels 
$f \colon \PP ' \to \PP$ induisant les morphismes
$a\colon X' \to X$ et $b \colon Y'\to Y$.

\item Un $d$-cadre localement propre est la donnée 
d'une immersion ouverte de $\V$-schémas formels lisses
$\PP \subset \QQ$ avec $\QQ$ propre, d'une immersion ouverte de $k$-schémas 
$Y \subset X$,
d'une immersion fermée 
$X \hookrightarrow \PP$,
d'un diviseur $T$ de $P$ tel que $Y = X \setminus T$. 
On note  $(\QQ, \PP, T, X, Y)$ un tel $d$-cadre localement propre. 

Un morphisme de $d$-cadres localement propres 
$(f, g, a,b) \colon (\QQ', \PP', T',X', Y') \to (\QQ,\PP,  T, X, Y)$ 
est un morphisme de $\V$-schémas formels 
$f \colon \QQ ' \to \QQ$ induisant les morphismes
$g \colon \PP ' \to \PP$, $a\colon X' \to X$ et $b \colon Y'\to Y$.
\end{enumerate}
\end{defi}

\begin{defi}
\label{defi-cadre-comp}
On définit respectivement la catégorie des {\og cadres \fg}, des
{\og cadres localement propres  \fg} et
{\og cadres propres  \fg}
de la manière suivante. 
\begin{enumerate}
\item Un  cadre  est la donnée 
d'un $\V$-schéma formel séparé lisse $\PP$, 
d'une immersion ouverte de $k$-schémas 
$Y \subset X$ et d'une immersion fermée 
$X \hookrightarrow \PP$. 
On note  $(Y, X, \PP)$ un tel cadre. 

Un morphisme de cadres  
$u=(b,a,f) \colon (Y', X', \PP') \to(Y, X, \PP)$ 
est un morphisme de $\V$-schémas formels 
$f \colon \PP ' \to \PP$ induisant les morphismes
$a\colon X' \to X$ et $b \colon Y'\to Y$.

\item Un  cadre localement propre  est la donnée 
d'une immersion ouverte de $\V$-schémas formels lisses
$\PP \subset \QQ$ avec $\QQ$ propre, d'une immersion ouverte de $k$-schémas 
$Y \subset X$ et d'une immersion fermée 
$X \hookrightarrow \PP$. 
On note  $(Y, X, \PP, \QQ)$ un tel cadre localement propre . 

Un morphisme de cadres localement propres  
$u=(b,a,g,f) \colon (Y', X', \PP', \QQ') \to(Y, X, \PP, \QQ)$ 
est un morphisme de $\V$-schémas formels 
$f \colon \QQ ' \to \QQ$ induisant les morphismes
$g \colon \PP ' \to \PP$, $a\colon X' \to X$ et $b \colon Y'\to Y$.

\item Un cadre propre  est la donnée d'un $\V$-schéma formel propre et lisse $\PP$
et d'une immersion $Y \hookrightarrow \PP$. On note $(Y, \PP)$ un tel cadre propre.
Un morphisme de cadres propres    
$u=(b,g) \colon (Y', \PP') \to(Y, \PP)$ 
est un morphisme de $\V$-schémas formels 
$g \colon \PP ' \to \PP$ induisant le morphisme
$b \colon Y'\to Y$.
\end{enumerate}
\end{defi}

\begin{rema}
On dispose d'un foncteur canonique pleinement fidèle de la catégorie des $d$-cadres dans celle des cadres.
La  notion de cadre de \ref{defi-cadre-comp} est mieux adapté aux complexes en général (voir par exemple \ref{rema-cadvsd-cad}).
Pour les distinguer nous avons ajouté la précision {\og $d$-cadres \fg} au lieu de {\og cadres \fg}
et nous avons inverser l'ordre de leur écriture. 
\end{rema}

\begin{nota}
On note $\mathfrak{Cad}$ la catégorie des cadres localement propres
et
$\mathfrak{Cadp} $ la catégorie des cadres propres.
On dispose du foncteur canonique pleinement fidèle 
$\mathfrak{Cadp} \to \mathfrak{Cad}$
défini par 
$(Y, \PP) \mapsto (Y, X, \PP, \PP)$, où $X$ est l'adhérence de $Y$ dans $P$.  

\end{nota}

\begin{nota}
\label{nota-FDcadre}
\begin{itemize}
\item Soit $\PP$ un $\V$-schéma formel lisse. 
On désigne par $F\text{-}D ^\mathrm{b} _\mathrm{surhol} (\D ^\dag _{\PP,\Q})$ la sous-catégorie pleine de 
$F\text{-}D ^\mathrm{b}  (\D ^\dag _{\PP,\Q})$ des 
 $F\text{-}$complexes surholonomes (voir \cite{caro_surholonome}).

\item Soit $\PP$ un $\V$-schéma formel lisse et $Y$ un sous-schéma de $P$. 
On note 
$F\text{-}D ^\mathrm{b} _\mathrm{surhol}  (Y, \PP/K)$ la sous-catégorie pleine de 
$F\text{-}D ^\mathrm{b} _\mathrm{surhol} (\D ^\dag _{\PP,\Q})$ 
des $F\text{-}$complexes $\E$ 
tels qu'il existe un isomorphisme de la forme 
$\R \underline{\Gamma} ^\dag _{Y} (\E)\riso \E$.

\item Soit $(Y, X, \PP)$ un cadre (resp. $(Y, X, \PP, \QQ)$ un cadre localement propre).
On note alors  
$F\text{-}D ^\mathrm{b} _\mathrm{surhol}  (Y, X, \PP/K):= F\text{-}D ^\mathrm{b} _\mathrm{surhol}  (Y, \PP/K)$
(resp. $F\text{-}D ^\mathrm{b} _\mathrm{surhol}  (Y, X, \PP,\QQ/K):= F\text{-}D ^\mathrm{b} _\mathrm{surhol}  (Y, \PP/K)$).

\item Soit $(\QQ,  \PP, T,X, Y)$ (resp. $(\PP, T ,X, Y)$) 
un $d$-cadre localement propre
(resp. un $d$-cadre). On note 
$F\text{-}\mathrm{Surhol}  (\PP,Y/K)$ la catégorie des
$F\text{-}\D ^\dag _{\PP,\Q}$-module surholonome $\E$ 
tels qu'il existe un isomorphisme de la forme 
$\R \underline{\Gamma} ^\dag _{Y} (\E)\riso \E$.
On posera aussi 
$F\text{-}\mathrm{Surhol} 
(\QQ,\PP, T,X, Y)
:=
F\text{-}\mathrm{Surhol}  (\PP,Y/K)$
(resp.
$F\text{-}\mathrm{Surhol} 
(\PP, T, X, Y)
:=
F\text{-}\mathrm{Surhol}  (\PP,Y/K)$.

\item Comme les propriétés de $F$-surholonomie et de $F$-surcohérence sont égales, 
on peut remplacer {\og$\mathrm{surhol}$\fg}
par {\og$\mathrm{surcoh}$\fg} dans les notations ci-dessus sans changer la catégorie. 
\end{itemize}
 
\end{nota}

\begin{vide}
Soit $(Y, X, \PP)$ un cadre . On note $\DD _{\PP}$ le dual 
$\D ^\dag _{\PP,\Q}$-linéaire (e.g., voir \cite{virrion} ou \cite{Beintro2}).
On dispose du foncteur dual, noté $\DD _{Y, \PP}\colon F\text{-}D ^\mathrm{b} _\mathrm{surhol}  (Y, X, \PP/K)
\to F\text{-}D ^\mathrm{b} _\mathrm{surhol}  (Y, X, \PP/K)$ défini en posant, pour tout objet 
$\E \in  F\text{-}D ^\mathrm{b} _\mathrm{surhol}  (Y, X, \PP/K)$, 
$$\DD _{Y, \PP}( \E) :=\R \underline{\Gamma} ^\dag _{Y} \circ \DD _{\PP} (\E).$$
On pourra simplement noté $\DD _Y$ ce foncteur s'il n'y a pas d'ambiguïté sur $\PP$. 
\end{vide}

\begin{rema}
\label{rema-dualZ}
Soient
$(Y, X, \PP)$ un cadre , 
$\E \in  F\text{-}D ^\mathrm{b} _\mathrm{surhol}  (Y, X, \PP/K)$
et
$Z$ un fermé de $P$ tel que $Z \cap X = X \setminus Y$.
Comme 
$\DD _{\PP} (\E) \in F\text{-}D ^\mathrm{b} _\mathrm{surhol}  (X, X, \PP/K)$, 
comme les foncteurs $ (\hdag Z ) $ et 
$\R \underline{\Gamma} ^\dag _{Y}$ sont canoniquement égaux sur $F\text{-}D ^\mathrm{b} _\mathrm{surhol}  (X, X, \PP/K)$
(voir \cite[3.2.1]{caro-2006-surcoh-surcv}),
on obtient alors l'isomorphisme canonique
$\DD _{Y, \PP}( \E) \riso (\hdag Z) \circ \DD _{\PP} (\E)$.
\end{rema}

\begin{lemm}
\label{bidualite}
Soit $(Y, X, \PP)$ un cadre .
On dispose, pour tout $\E \in  F\text{-}D ^\mathrm{b} _\mathrm{surhol}  (Y, X, \PP/K)$, 
de l'isomorphisme de bidualité: 
\begin{equation}
\label{bidualite-iso}
\DD _{Y, \PP} \circ \DD _{Y, \PP} (\E) \riso \E.
\end{equation}
\end{lemm}

\begin{proof}
Notons $Z:= X \setminus Y$. Comme 
$\DD _{\PP} \circ \R \underline{\Gamma} ^\dag _{Z} \circ \DD _{\PP}(\E) $ est à support dans $Z$
alors 
$(\hdag Z) \circ \DD _{\PP} \circ \R \underline{\Gamma} ^\dag _{Z} \circ \DD _{\PP}(\E) =0$.
On en déduit le premier isomorphisme:
$$\DD _{Y, \PP} \circ \DD _{Y, \PP} (\E)
=
(\hdag Z) \circ \DD _{\PP} \circ (\hdag Z)  \circ \DD _{\PP}(\E) 
\riso 
(\hdag Z) \circ \DD _{\PP}  \circ \DD _{\PP}(\E) 
\riso 
(\hdag Z) (\E) 
\riso 
\E.$$
\end{proof}

\begin{vide}
Soit $(Y, X, \PP)$ un cadre .
D'après \cite{caro-stab-prod-tens},
le bifoncteur 
$-
\smash{\overset{\L}{\otimes}}   ^{\dag}
_{\O  _{\PP,\Q}}
-
[d _{Y/P}]
\colon 
F\text{-}D ^\mathrm{b} _\mathrm{surhol} (\D ^\dag _{\PP,\Q})
\times
F\text{-}D ^\mathrm{b} _\mathrm{surhol} (\D ^\dag _{\PP,\Q})
\to 
F\text{-}D ^\mathrm{b} _\mathrm{surhol} (\D ^\dag _{\PP,\Q})
$
se factorise en 
le bifoncteur produit tensoriel que l'on notera 
$$-
\smash{\overset{\L}{\otimes}}   ^{\dag}
_{\O  _{(Y, \PP)}}
-
\colon 
F\text{-}D ^\mathrm{b} _\mathrm{surhol}  (Y, X, \PP/K)
\times
F\text{-}D ^\mathrm{b} _\mathrm{surhol}  (Y, X, \PP/K)
\to 
F\text{-}D ^\mathrm{b} _\mathrm{surhol}  (Y, X, \PP/K).$$
\end{vide}

\begin{vide}
Soit $u=(b,a,f) \colon (Y', X', \PP') \to(Y, X, \PP)$ 
un morphisme de cadres .
On définit le foncteur  
$u ^!\colon 
F\text{-}D ^\mathrm{b} _\mathrm{surhol}  (Y, X, \PP/K)
\to 
F\text{-}D ^\mathrm{b} _\mathrm{surhol}  (Y', X', \PP'/K)
$
image inverse extraordinaire par $u$ en posant
$u ^!  := \R \underline{\Gamma} ^\dag _{Y'}  \circ f ^{!}$.
On en déduit comme d'habitude le foncteur  
$u ^+\colon 
F\text{-}D ^\mathrm{b} _\mathrm{surhol}  (Y, X, \PP/K)
\to 
F\text{-}D ^\mathrm{b} _\mathrm{surhol}  (Y', X', \PP'/K)
$
image inverse par $u$ en posant
$u ^+  := \DD _{Y', \PP'} \circ u ^{!}\circ \DD _{Y, \PP}$.
On remarque que ces foncteurs ne dépendent pas des fermés $X$ et $X'$.

\end{vide}

\begin{vide}
Soit $u=(b,a,f) \colon (Y', X', \PP') \to(Y, X, \PP)$ 
un morphisme de cadres  tel que $a$ soit propre.
On définit le foncteur  
$u _+\colon 
F\text{-}D ^\mathrm{b} _\mathrm{surhol}  (Y', X', \PP'/K)
\to 
F\text{-}D ^\mathrm{b} _\mathrm{surhol}  (Y, X, \PP/K)$
image directe par $u$ en posant
$u _+ := f _+$.
On définit le foncteur  
$u _!\colon 
F\text{-}D ^\mathrm{b} _\mathrm{surhol}  (Y', X', \PP'/K)
\to 
F\text{-}D ^\mathrm{b} _\mathrm{surhol}  (Y, X, \PP/K)$
image directe extraordinaire par $u$ en posant
$u _! := \DD _{Y, \PP} \circ f _+\circ \DD _{Y', \PP'}$.
On remarque que ces foncteurs ne dépendent pas des fermés $X$ et $X'$. 
\end{vide}

\begin{vide}
Soient $u=(b,a,g,f) \colon (Y', X', \PP', \QQ') \to(Y, X, \PP, \QQ)$ 
un morphisme de cadres localement propres 
et $v= (b, a, g)$ le morphisme de cadres  induits.
Avec les notations de \ref{nota-FDcadre},
On définit l'image inverse extraordinaire par $u$ notée 
$u ^{!}\colon F\text{-}D ^\mathrm{b} _\mathrm{surhol}  (Y, X, \PP,\QQ/K)
\to
F\text{-}D ^\mathrm{b} _\mathrm{surhol}  (Y', X', \PP',\QQ'/K)$
en posant $u ^{!} = v ^{!}$. 
De même, on définit le foncteur image inverse par $u$ en posant
$u ^{+} = v ^{+}$.
Lorsque $a$ est propre, on dispose des image directe et image directe extraordinaire par $u$ 
définis en posant
$u _{+} = v _{+}$,
$u _{!} = v _{!}$.
On remarque que ces foncteurs ne dépendent pas des fermés $X$ et $X'$ ni de $\QQ$ et $\QQ'$. 

\end{vide}

\begin{lemm}
\label{cgtdebase}
Soient 
$u=(b,a,g,f) \colon (Y', X', \PP', \QQ') \to(Y, X, \PP, \QQ)$
et 
$\theta =(d,c,i,h) \colon (\widetilde{Y}, \widetilde{X}, \widetilde{\PP}, \widetilde{\QQ}) \to(Y, X, \PP, \QQ)$
deux morphismes de cadres localement propres 
avec $h$ lisse et  $c$ propre. 
Soient $(\widetilde{Y}', \widetilde{X}', \widetilde{\PP}', \widetilde{\QQ}'):=
(\widetilde{Y} \times _Y Y ', \widetilde{X}\times _X X ', \widetilde{\PP} \times _{\PP} \PP ' , \widetilde{\QQ}\times _{\QQ} \QQ ')$
et
$\theta '= (d',c',i',h') \colon (\widetilde{Y}', \widetilde{X}', \widetilde{\PP}', \widetilde{\QQ}') \to(Y', X', \PP', \QQ')$, 
$\widetilde{u}= (\widetilde{b}, \widetilde{a},\widetilde{g},\widetilde{f})
\colon (\widetilde{Y}', \widetilde{X}', \widetilde{\PP}', \widetilde{\QQ}') \to (\widetilde{Y}, \widetilde{X}, \widetilde{\PP}, \widetilde{\QQ}) $
les projections canoniques. 
On dispose pour tout 
$\widetilde{\E} \in F\text{-}D ^\mathrm{b} _\mathrm{surhol}  (\widetilde{Y}, \widetilde{\PP}/K)$
de l'isomorphisme canonique
\begin{equation}
\label{iso-chgt-base-cadre}
u ^! \circ \theta _+ (\widetilde{\E})
\riso
\theta ' _+ \circ \widetilde{u} ^{!}
(\widetilde{\E}).
\end{equation}

\end{lemm}

\begin{proof}
Par définition, 
$\theta ' _+ \circ \widetilde{u} ^{!}
(\widetilde{\E})=
i ' _+ \circ \R \underline{\Gamma} ^\dag _{\widetilde{Y}  '}  \circ   \widetilde{g} ^!  (\widetilde{\E})$.
Grâce à l'isomorphisme de changement de base commutant à Frobenius (voir \cite{Abe-Frob-Poincare-dual}),
$u ^! \circ \theta _+ (\widetilde{\E}) =
\R \underline{\Gamma} ^\dag _{Y '} \circ g ^! \circ  i _+ (\widetilde{\E})
\riso 
\R \underline{\Gamma} ^\dag _{Y '}  \circ  i ' _+ \circ \widetilde{g} ^!  (\widetilde{\E})
\riso
i ' _+ \circ \R \underline{\Gamma} ^\dag _{\widetilde{\PP} \times _{\PP} Y '}  \circ   \widetilde{g} ^!  (\widetilde{\E})$.
Comme $\R \underline{\Gamma} ^\dag _{\widetilde{Y}}  (\widetilde{\E})
\riso
\widetilde{\E}$, on en déduit 
$\R \underline{\Gamma} ^\dag _{\widetilde{Y} \times _{\PP} \PP '}  \circ   \widetilde{g} ^!  (\widetilde{\E})
\riso 
  \widetilde{g} ^!  (\widetilde{\E})$.
Comme  
$(\widetilde{\PP} \times _{\PP} Y ' )\cap( \widetilde{Y} \times _{\PP} \PP ')=
\widetilde{Y}  '$, 
il en résulte 
$  \R \underline{\Gamma} ^\dag _{\widetilde{\PP} \times _{\PP} Y '}  \circ   \widetilde{g} ^!  (\widetilde{\E})
\riso 
\R \underline{\Gamma} ^\dag _{\widetilde{Y}  '}  \circ   \widetilde{g} ^!  (\widetilde{\E})$.
D'où le résultat.
\end{proof}

\section{Catégories de complexes de type surcohérent}

\begin{defi}
\label{defi-couples}
$\bullet$ La catégorie, notée $\mathfrak{Cpl}$, des couples (de $k$-variétés) a pour objets
les couples $(Y,X)/K$ avec $X$ une $k$-variété et $Y$ est un ouvert de $X$
et pour morphismes les morphismes de $k$-variétés 
$a \colon X' \to X$ tels que $a (Y') \subset Y$. On notera $(b,a)\colon (Y', X' ) \to (Y,X)$ un tel morphisme, où
$b\colon Y'\to Y$ désigne le morphisme un induit par $a$.
On dira qu'un morphisme $(b,a)$ de $\mathfrak{Cpl}$ est {\og complet\fg} si $a$ est propre. 
\end{defi}

\begin{defi}
\label{uni}
Soit $(Y,X)/K$ un couple.
On note $\mathfrak{Uni} (Y,X/K)$ 
la catégorie des cadres localement propres  au-dessus de $(Y,X)$.
Un objet de $\mathfrak{Uni} (Y,X/K)$ est ainsi la donnée d'un cadre localement propre  $(Y', X', \PP', \QQ')$  et d'un morphisme de 
$\mathfrak{Cpl}$
de la forme $(b,a) \colon (Y',X') \to (Y,X)$. 
On notera $(b,a) \colon (Y', X', \PP', \QQ') \to (Y,X)$ un tel objet ou plus simplement par abus de notations $(Y', X', \PP', \QQ')$.
Les morphismes 
$((Y'', X'', \PP'', \QQ''), (b',a') )
\to
((Y', X', \PP', \QQ'), (b,a) )$
de $\mathfrak{Uni} (Y,X/K)$ sont les morphismes de $\mathfrak{Cad}$ de la forme 
$u=(d,c, g,f)\colon (Y'', X'', \PP'', \QQ'')
\to
(Y', X', \PP', \QQ')$
tels que 
$(b,a) \circ (d,c) =(b',a') $.
On pourra noter abusivement 
$u\colon (Y'', X'', \PP'', \QQ'')
\to
(Y', X', \PP', \QQ')$ un tel morphisme. 
\end{defi}

La définition qui suit s'inspire fortement de la procédure analogue donnée en cohomologie rigide (voir \cite[7.3.7]{LeStum-livreRigCoh}):
\begin{defi}
\label{defi-surcoh-cadre}
Soit $(Y,X)/K$ un couple.
On définit la catégorie $F\text{-}D ^\mathrm{b} _\mathrm{surcoh} (\D ^\dag _{(Y,X)/K})$ des complexes de type surcohérent sur $(Y,X)/K$
de la manière suivante:
\begin{itemize}
\item Un objet est la donnée 
\begin{itemize}
\item  d'une famille d'objets $\E _{(Y', X', \PP', \QQ')} $ de 
$F\text{-}D ^\mathrm{b} _\mathrm{surcoh}  (Y', X', \PP',\QQ'/K)$, 
où $(Y', X', \PP', \QQ')$ parcourt les objets de $\mathfrak{Uni} (Y,X/K)$ ;
\item pour toute flèche
$u\colon  (Y'', X'', \PP'', \QQ'') \to(Y', X', \PP', \QQ')$ de $\mathfrak{Uni} (Y,X/K)$,
d'un isomorphisme 
$$\phi _u \colon u ^{!} (\E _{(Y', X', \PP', \QQ')} ) \riso \E _{(Y'', X'', \PP'', \QQ'')} $$ 
dans $F\text{-}D ^\mathrm{b} _\mathrm{surcoh}  (Y'', X'', \PP'',\QQ''/K)$,
ces isomorphismes vérifiant la condition de cocycle : pour tous morphismes 
$u\colon  (Y'', X'', \PP'', \QQ'') \to(Y', X', \PP', \QQ')$ et
$v\colon  (Y''', X''', \PP''', \QQ''') \to(Y'', X'', \PP'', \QQ'')$
de $\mathfrak{Uni} (Y,X/K)$,
le diagramme
\begin{equation}
\xymatrix @R=0,3cm{
 {v ^{!} \circ u ^{!} (\E _{(Y', X', \PP', \QQ')} ) } 
 \ar[r] ^-{v ^{!} (\phi _u)}
 \ar[d] ^-{\sim}
 & 
 {v ^{!} ( \E _{(Y'', X'', \PP'', \QQ'')} )} 
  \ar[d] ^-{\phi _v}
 \\ 
 {(u\circ v) ^{!} (\E _{(Y', X', \PP', \QQ')} ) } 
  \ar[r] ^-{\phi _{u\circ v}}
 & 
 {\E _{(Y''', X''', \PP''', \QQ''')} } 
 }
\end{equation}

soit commutatif.
\end{itemize}
On notera de manière elliptique 
$(\E _{(Y', X', \PP', \QQ')} , \phi _u )$ une telle donnée.

\item Un morphisme $\alpha\colon (\E _{(Y', X', \PP', \QQ')} , \phi _u ) 
\to 
(\FF _{(Y', X', \PP', \QQ')} , \psi _u )$
de
$F\text{-}D ^\mathrm{b} _\mathrm{surcoh} (\D ^\dag _{(Y,X)/K})$ 
est la donnée d'une famille 
de morphismes
$\alpha _{(Y', X', \PP', \QQ')}\colon \E _{(Y', X', \PP', \QQ')} \to \FF _{(Y', X', \PP', \QQ')} $
 de 
$F\text{-}D ^\mathrm{b} _\mathrm{surcoh}  (Y', X', \PP',\QQ'/K)$
telle que pour tout morphisme
$u\colon  (Y'', X'', \PP'', \QQ'') \to(Y', X', \PP', \QQ')$ 
de 
$\mathfrak{Uni} (Y,X/K)$ on ait 
$\psi _u \circ u ^{!}(\alpha _{(Y', X', \PP', \QQ')})
=
\alpha _{(Y'', X'', \PP'', \QQ'')} \circ \phi _u $.
\end{itemize}

\end{defi}

\begin{lemm}
Soit $(Y, X, \PP)$ un cadre. 
Soient
$pr _1 \colon (Y, X, \PP\times \PP) \to (Y, X, \PP)$
et
$pr _2 \colon (Y, X, \PP\times \PP) \to (Y, X, \PP)$
les morphismes induits respectivement par projection à gauche et à droite. 
Pour tout 
$\E\in F\text{-}D ^\mathrm{b} _\mathrm{surcoh}  (Y, X, \PP/K)$, 
on dispose de l'isomorphisme canonique dans 
$F\text{-}D ^\mathrm{b} _\mathrm{surcoh}  (Y, X, \PP\times \PP/K)$
de la forme :
\begin{equation}
\label{pr2=pr1}
pr _2 ^{!} (\E) \riso pr _1 ^{!} (\E).
\end{equation}

\end{lemm}

\begin{proof}
Notons $\delta = (id, id, d) \colon (Y, X, \PP) \to  (Y, X, \PP\times \PP)$
le morphisme induit par l'immersion fermée diagonale. 
Comme d'après le théorème de Berthelot-Kashiwara 
le foncteur $\delta ^{!}=d ^{!}$ est pleinement fidèle sur 
$ F\text{-}D ^\mathrm{b} _\mathrm{surcoh}  (Y, X, \PP\times \PP/K)$, 
il s'agit de valider 
un isomorphisme canonique de la forme
$$\delta ^{!} \circ pr _2 ^{!} (\E) \riso \delta ^{!} \circ pr _1 ^{!} (\E),$$
ce qui découle de la transitivité des images inverses extraordinaires.
\end{proof}

\begin{lemm}
\label{4.2.3.4}
Soit 
$u=(b,a,g,f) \colon (Y, X', \PP', \QQ') \to(Y, X, \PP, \QQ)$ 
un morphisme de cadres localement propres 
tel que $a$ soit propre et $f$ soit lisse.
\begin{enumerate}
\item On dispose d'isomorphismes canoniques de la forme 
$u _{+} \circ u ^{!} \riso id$ et 
$id \riso u ^{!} \circ u _+$. En particulier, 
les foncteurs $u _+$ et $u ^!$ induisent des équivalences quasi-inverses entre 
$F\text{-}D ^\mathrm{b} _\mathrm{surcoh}  (Y, X', \PP',\QQ'/K)$
et
$F\text{-}D ^\mathrm{b} _\mathrm{surcoh}  (Y, X, \PP,\QQ/K)$.

\item De plus, on dispose des isomorphismes canoniques 
$u _{+} \riso u _!$ et 
$u ^{!} \riso u ^{+}$.
\end{enumerate}

\end{lemm}

\begin{proof}
0) Posons $Z := X \setminus Y$ et $Z ':= g ^{-1} (Z)$. 
Comme les foncteurs $u _+$, $u _!$, $u ^!$ et $u ^+$ ne dépendent
pas de $X'$, on peut supposer que $Y$ est dense dans $X'$.
Comme $a$ est propre, on en déduit alors $Y = a ^{-1} (Y)$. 
Il en résulte $X ' \cap Z' = X' \setminus Y$.
D'après \ref{rema-dualZ}, on obtient alors les isomorphismes canoniques
$\DD _{Y, \PP'}  \riso (\hdag Z') \circ \DD _{\PP'}$
et
$\DD _{Y, \PP}  \riso (\hdag Z) \circ \DD _{\PP}$.

1) Comme $a$ est propre, 
le morphisme $u$ est le composé des morphismes de cadres localement propres 
$(Y, X', \PP', \QQ') \to (Y, X', f ^{-1}(\PP), \QQ') \to(Y, X, \PP, \QQ)$.
On se ramène ainsi à l'un de ces deux cas. 
Pour le premier cas, on dispose de l'isomorphisme canonique $u ^{!} \riso u ^{+}$.
Dans le second cas, comme $g$ est alors propre, 
on bénéficie de théorème de dualité relative:
$u _!= (\hdag Z) \circ \DD _{\PP}  \circ g _{+} \circ  (\hdag Z')\circ  \DD _{\PP'}
\riso
(\hdag Z) \circ \DD _{\PP}  \circ (\hdag Z')\circ g _{+} \circ   \DD _{\PP'}
\riso
(\hdag Z) \circ \DD _{\PP}  \circ (\hdag Z') \circ   \DD _{\PP'}\circ g _{+}
\riso g _+ = u _+$.
Il suffit alors de vérifier le premier point du lemme, ce qui  
se vérifie de manière analogue à \cite[4.2.3.4]{caro-image-directe}:
pour tout $\E \in F\text{-}D ^\mathrm{b} _\mathrm{surcoh}  (Y, X, \PP,\QQ/K)$
il existe $Y =\sqcup _{i=1,\dots , r} Y _i$
une $d$-stratification lisse de $Y$ dans $P$ au-dessus de laquelle 
$\E $ se dévisse en isocristaux surconvergents (voir les définitions et résultats de \cite[4.1]{caro-stab-prod-tens}).
Comme $f$ est lisse, 
alors  $Y =\sqcup _{i=1,\dots , r} Y _i$ est aussi
une $d$-stratification lisse de $Y$ dans $P'$.
Afin d'obtenir l'isomorphisme canonique 
$u _{+} \circ u ^{!}(\E)  
\riso 
\E$, 
on se ramène alors par dévissage à la situation de \cite[4.2.3.3]{caro-image-directe}.
De même, 
pour tout $\E '\in F\text{-}D ^\mathrm{b} _\mathrm{surcoh}  (Y, X', \PP',\QQ'/K)$
il existe $Y =\sqcup _{i=1,\dots , r} Y '_i$
une $d$-stratification lisse de $Y$ dans $P'$ au-dessus de laquelle 
$\E '$ se dévisse en isocristaux surconvergents.
Quitte à rétrécir cette stratification, on peut en outre supposer que
$Y =\sqcup _{i=1,\dots , r} Y '_i$
est une $d$-stratification lisse de $Y$ dans $P$.
Afin d'obtenir l'isomorphisme canonique 
$u ^{!} \circ u _{+} (\E')  
\riso 
\E'$,
on se ramène alors par dévissage à la situation de \cite[4.2.3.3]{caro-image-directe}.
\end{proof}

\begin{prop}
\label{eqcat-YXvsP}
Soit $(Y, X, \PP, \QQ)$ un cadre localement propre .
Le foncteur canonique de restriction 
$F\text{-}D ^\mathrm{b} _\mathrm{surcoh} (\D ^\dag _{(Y,X)/K})
\to 
F\text{-}D ^\mathrm{b} _\mathrm{surcoh}  (Y, X, \PP,\QQ/K)$
est une équivalence de catégories.
\end{prop}

\begin{proof}
La preuve est analogue à celle de \cite[7.3.11]{LeStum-livreRigCoh}: 
on construit un foncteur canonique quasi-inverse de la manière suivante. 
Soit $\E \in F\text{-}D ^\mathrm{b} _\mathrm{surcoh}  (Y, X, \PP,\QQ/K)$. 
Pour tout objet $(Y', X', \PP', \QQ')$ de $\mathfrak{Uni} (Y,X/K)$, on pose 
$pr _2 \colon (Y', X', \PP'\times \PP , \QQ' \times \QQ) \to (Y, X, \PP, \QQ)$
et
$pr _1 \colon (Y', X', \PP'\times \PP , \QQ' \times \QQ) \to (Y', X', \PP', \QQ')$
les morphismes canoniques de 
$\mathfrak{Uni} (Y,X/K)$.
On obtient alors canoniquement un objet de 
$F\text{-}D ^\mathrm{b} _\mathrm{surcoh}  (Y', X', \PP',\QQ'/K)$
en posant 
$$\E _{(Y', X', \PP', \QQ')} := pr _{1+} \circ pr _2 ^{!} (\E).$$
Soit 
$u=(b,a,g,f) \colon  (Y'', X'', \PP'', \QQ'') \to(Y', X', \PP', \QQ')$ un morphisme de $\mathfrak{Uni} (Y,X/K)$.
Notons 
$pr '_1 \colon (Y'', X'', \PP''\times \PP ', \QQ'' \times \QQ'') \to (Y'', X'', \PP'', \QQ'')$
et
$pr '_2 \colon (Y'', X'', \PP''\times \PP , \QQ'' \times \QQ) \to (Y, X, \PP, \QQ)$
les morphismes canoniques
et
$v= (b,a,g\times id,f\times id ) 
\colon 
(Y'', X'', \PP''\times \QQ, \QQ''\times \QQ) \to(Y', X', \PP'\times \QQ, \QQ'\times \QQ)$.
On dispose de l'isomorphisme canonique
\begin{equation}
\label{cgtbase-u-pre}
pr  _{1} ^{\prime!} \circ u ^{!} \circ pr _{1+} 
\riso
v ^{!} \circ pr  _{1} ^{!} \circ pr _{1+} 
\underset{\ref{4.2.3.4}}{\riso} 
v ^{!} 
\underset{\ref{4.2.3.4}}{\riso}
pr  _{1} ^{\prime !} \circ pr '_{1+} \circ v ^{!} .
\end{equation}
Comme le foncteur 
$pr _{1} ^{\prime !}$ 
est pleinement fidèle, 
on déduit de \ref{cgtbase-u-pre}
l'isomorphisme canonique
$c _u \colon u ^{!} \circ pr _{1+} 
\riso 
pr ' _{1+} \circ v ^{!}$
tel que 
$pr  _{1} ^{\prime!} (c _u)$ soit  égale au composé de \ref{cgtbase-u-pre}.
Ces isomorphismes $c _u$ sont transitifs, i.e.,
si 
$u'\colon  (Y''', X''', \PP''', \QQ''') \to(Y'', X'', \PP'', \QQ'')$ 
est un second morphisme de $\mathfrak{Uni} (Y,X/K)$,
en notant $pr '' _{1}$, $pr '' _{2}$, $v '$ les morphismes comme ci-dessus mais avec des primes en plus,
on dispose alors du diagramme commutatif suivant
\begin{equation}
\label{trans-cu}
\xymatrix @R=0,3cm{
{(u \circ u ') ^! \circ pr  _{1+}} 
\ar[rr] ^-{\sim} _{c _{u \circ u '}}
\ar[d] ^-{\sim}
&& 
{pr '' _{1+} \circ (v \circ v ') ^!} 
\ar[d] ^-{\sim}
\\ 
{u ^{\prime !} \circ u ^! \circ pr  _{1+}} 
\ar[r] ^-{\sim} _{u ^{\prime !} \circ c _{u }}
& 
{u ^{\prime !}  \circ pr ' _{1+}\circ v ^! } 
\ar[r] ^-{\sim} _{ c _{u'} \circ v ^{ !}}
&
{pr '' _{1+} \circ v ^{\prime !} \circ v ^!.} 
}
\end{equation}
En effet, il suffit de vérifier que l'image par 
$pr _{1} ^{\prime \prime!}$ du diagramme \ref{trans-cu} est commutatif. 
Pour cela, 
considérons le diagramme ci-dessous:
\begin{equation}
\notag
\xymatrix @R=0,3cm @C=2cm {
{pr _{1} ^{\prime \prime!}\circ (u \circ u ') ^! \circ pr  _{1+}} 
\ar[rr] ^-{\sim} _{pr _{1} ^{\prime \prime!}(c _{u \circ u '})}
\ar[d] ^-{\sim}
&& 
{pr _{1} ^{\prime \prime!}\circ pr '' _{1+} \circ (v \circ v ') ^!} 
\ar[d] ^-{\sim}
\\
{pr _{1} ^{\prime \prime!}\circ u ^{\prime !} \circ u ^! \circ pr  _{1+}} 
\ar[r] ^-{\sim} _{pr _{1} ^{\prime \prime!}\circ u ^{\prime !} \circ c _{u }}
\ar[d] ^-{\sim}
& 
{pr _{1} ^{\prime \prime!}\circ u ^{\prime !}  \circ pr ' _{1+}\circ v ^! } 
\ar[r] ^-{\sim} _{pr _{1} ^{\prime \prime!}\circ  c _{u'} \circ v ^{ !}}
\ar[d] ^-{\sim}
&
{pr _{1} ^{\prime \prime!}\circ pr '' _{1+} \circ v ^{\prime !} \circ v ^!} 
\ar[d] ^-{\sim} _-{\ref{4.2.3.4}}
\\
{v ^{\prime !} \circ pr _{1} ^{\prime!}\circ  u ^! \circ pr  _{1+}} 
\ar[r] ^-{\sim} _{v ^{\prime !} \circ pr _{1} ^{\prime!} \circ c _{u }}
\ar[d] ^-{\sim}
& 
{v ^{\prime !} \circ pr _{1} ^{\prime!} \circ pr ' _{1+}\circ v ^! } 
\ar[r] ^-{\sim}   _-{\ref{4.2.3.4}}
\ar[d] ^-{\sim}  _-{\ref{4.2.3.4}}
&
{v ^{\prime !} \circ v ^!} 
\ar[d] ^-{\sim}
\\ 
{v ^{\prime !} \circ  v ^! \circ pr _{1} ^{!}\circ  pr  _{1+}} 
\ar[r] ^-{\sim}  _-{\ref{4.2.3.4}}
& 
{v ^{\prime !} \circ v ^!} 
\ar[r] ^-{\sim} 
&
{ (v \circ v ') ^!,} 
}
\end{equation}
dont les isomorphismes non indiqués sont induits par transitivité des foncteurs images inverses extraordinaires.
Les carrés du milieu à droite, en bas à gauche et le grand contour sont commutatifs par définition (voir \ref{cgtbase-u-pre}).
La commutativité des carrés du milieu à gauche et en bas à droite est évidente. 
D'où le résultat.

On construit l'isomorphisme $\phi _u$ via la composition:
$$\phi _u \colon 
u ^{!}(\E _{(Y', X', \PP', \QQ')}) =
u ^{!} \circ pr _{1+} \circ pr _2 ^{!} (\E)
\underset{c _u}{\riso} 
pr ' _{1+} \circ v ^{!} \circ pr _2 ^{!} (\E)
\riso
pr ' _{1+} \circ pr _2 ^{\prime !} (\E)
=
\E _{(Y'', X'', \PP'', \QQ'')}.$$
Considérons le diagramme ci-dessous:
\begin{equation}
\notag
\xymatrix{
{u ^{\prime !} \circ u ^{!} \circ pr _{1+} \circ pr _2 ^{!} (\E)}
\ar[r] ^-{c _u} 
\ar[d] ^-{}
& 
{u ^{\prime !} \circ pr ' _{1+} \circ v ^{!} \circ pr _2 ^{!} (\E)} 
\ar[d] ^-{c _{u'}}
\ar[r] ^-{}
& 
{u ^{\prime !} \circ pr ' _{1+} \circ pr _2 ^{\prime !} (\E)} 
\ar[d] ^-{c _{u'}}
\\ 
{(u \circ u ^{\prime }) ^{!} \circ pr _{1+} \circ pr _2 ^{!} (\E)}
\ar[rd] ^-{c _{u\circ u'}}
& 
{pr '' _{1+} \circ v ^{\prime !} \circ  v ^{!} \circ pr _2 ^{!} (\E)} 
\ar[r] ^-{}
\ar[d] ^-{}
&
{pr '' _{1+} \circ  v ^{\prime !} \circ pr _2 ^{\prime !} (\E)} 
\ar[d] ^-{}
\\
{ }
&
{pr '' _{1+} \circ  (v  \circ  v ^{\prime }) ^{!} \circ pr _2 ^{!} (\E)} 
\ar[r] ^-{}
&
{pr '' _{1+} \circ pr _2 ^{\prime \prime !} (\E).} 
}
\end{equation}
Le trapèze est commutatif d'après \ref{trans-cu}, les carrés aussi. 
La commutativité de ce diagramme signifie que les isomorphismes $\phi _u $ satisfont à la condition de cocycle.
Notons alors  $\beta \colon 
F\text{-}D ^\mathrm{b} _\mathrm{surcoh}  (Y, X, \PP,\QQ/K)
\to 
F\text{-}D ^\mathrm{b} _\mathrm{surcoh} (\D ^\dag _{(Y,X)/K})$
le foncteur défini par $\E \mapsto (\E _{(Y', X', \PP', \QQ')} , \phi _u)$.

Vérifions à présent que $\beta $ est canoniquement quasi-inverse du foncteur restriction noté
$\alpha \colon F\text{-}D ^\mathrm{b} _\mathrm{surcoh} (\D ^\dag _{(Y,X)/K})
\to 
F\text{-}D ^\mathrm{b} _\mathrm{surcoh}  (Y, X, \PP,\QQ/K)$.
Soit $\E \in F\text{-}D ^\mathrm{b} _\mathrm{surcoh}  (Y, X, \PP,\QQ/K)$.
Avec les notations ci-dessus utilisées avec
$(Y', X', \PP', \QQ')= (Y, X, \PP, \QQ)$, on obtient les isomorphismes canoniques:
$$\alpha \circ \beta (\E)= \E _{(Y, X, \PP, \QQ)} = pr _{1+} \circ pr _2 ^{!} (\E)
\underset{\ref{pr2=pr1}}{\riso} 
pr _{1+} \circ pr _1 ^{!} (\E)
\riso \E.$$

Réciproquement soit 
$(\FF _{(Y', X', \PP', \QQ')} , \psi _u)
\in 
F\text{-}D ^\mathrm{b} _\mathrm{surcoh} (\D ^\dag _{(Y,X)/K})$.
Posons 
$\E:=\alpha (\FF _{(Y', X', \PP', \QQ')} , \phi _u)=
\FF _{(Y, X, \PP, \QQ)}$
et
$(\E _{(Y', X', \PP', \QQ')} , \phi _u):= 
\beta \circ \alpha (\E _{(Y', X', \PP', \QQ')} , \phi _u)
= 
\beta  (\E)$.
Comme le foncteur $pr _1 ^{!} $ est pleinement fidèle, il existe un unique
isomorphisme
$$\psi ^\sharp _{pr _1}\colon 
\FF _{(Y', X', \PP', \QQ')}
\riso 
pr _{1+}
(\FF _{(Y', X', \PP'\times \PP, \QQ'\times \QQ)})$$
tel que $pr _1 ^{!} (\psi ^\sharp  _{pr _1} ) $ soit l'isomorphisme composé:
$$pr _1 ^{!} (\psi ^\sharp _{pr _1} ) \colon 
pr _1 ^{!} (\FF _{(Y', X', \PP', \QQ')})
\underset{\psi _{pr _1}}{\riso}
\FF _{(Y', X', \PP'\times \PP, \QQ'\times \QQ)}
\riso 
pr _1 ^{!} \circ pr _{1+}
(\FF _{(Y', X', \PP'\times \PP, \QQ'\times \QQ)}).$$
On obtient alors l'isomorphisme:
\begin{equation}
\label{EisoF}
\E _{(Y', X', \PP', \QQ')} := pr _{1+} \circ pr _2 ^{!} (\E)
\underset{pr _{1+} ( \psi _{pr _2})}{\riso}
pr _{1+}  (\FF _{(Y', X', \PP'\times \PP, \QQ'\times \QQ)})
\underset{\psi ^\sharp _{pr _1}}{\liso}
\FF _{(Y', X', \PP', \QQ')}.
\end{equation}
Il reste à valider que la famille d'isomorphismes \ref{EisoF} commute au morphisme $\phi _u$ et $\psi _u$, i.e. que le diagramme
ci-dessous
\begin{equation}
\label{eqcat-YXvsP-diag}
\xymatrix  @R=0,3cm @C=2cm {
{u ^{!} \circ pr _{1+} \circ pr _2 ^{!} (\E)} 
\ar[r] ^-{u ^{!} \circ pr _{1+} ( \psi _{pr _2})} _-{\sim}
\ar[d] ^-{\sim} _-{c _u}
&
{u ^{!} \circ pr _{1+}  (\FF _{(Y', X', \PP'\times \PP, \QQ'\times \QQ)})} 
\ar[d] ^-{\sim} _-{c _u}
& 
{u ^{!} (\FF _{(Y', X', \PP', \QQ')})} 
\ar[l] ^-{u ^{!} \circ \psi ^\sharp _{pr _1}} _-{\sim}
\ar[dd] ^-{\psi _u} _-{\sim}
\\ 
{pr ' _{1+} \circ v ^{!} \circ  pr _2 ^{!} (\E)} 
\ar[r] ^-{pr ' _{1+} \circ v ^{!} ( \psi _{pr _2})} _-{\sim}
\ar[d] ^-{\sim} 
&
{pr ' _{1+} \circ v ^{!}  (\FF _{(Y', X', \PP'\times \PP, \QQ'\times \QQ)})} 
\ar[d] ^-{pr ' _{1+} ( \psi _v)} _-{\sim}
& 
{} 
\\
{pr ' _{1+} \circ pr _2 ^{\prime !} (\E)} 
\ar[r] ^-{pr ' _{1+} ( \psi _{pr '_2})} _-{\sim}
&
{pr '_{1+}  (\FF _{(Y'', X'', \PP''\times \PP, \QQ''\times \QQ)})} 
& 
{\FF _{(Y'', X'', \PP''\times \PP, \QQ''\times \QQ)}} 
\ar[l] ^-{\psi ^\sharp _{pr ' _1}} _-{\sim}
 }
\end{equation}
 est commutatif.
 Les carrés de gauche étant clairement commutatifs, il reste à vérifier celui de droite.
 Or, on vérifie que le contour du diagramme ci-dessous
\begin{equation}
\label{eqcat-YXvsP-diag-dcom}
\xymatrix{
{pr ^{\prime !} _1 \circ u ^{!} \circ pr _{1+}  (\FF _{(Y', X', \PP'\times \PP, \QQ'\times \QQ)})} 
\ar[d] ^-{\sim} 
\ar@/_8pc/[dd] ^-{\sim} _-{pr ^{\prime !} _1 (c _u)}
&& 
{pr ^{\prime !} _1 \circ u ^{!} (\FF _{(Y', X', \PP', \QQ')})} 
\ar[ll] ^-{pr ^{\prime !} _1 \circ u ^{!} \circ \psi ^\sharp _{pr _1}} _-{\sim}
\ar[d] ^-{\sim} 
\\ 
{v ^{!} \circ pr ^{!} _1 \circ pr _{1+}  (\FF _{(Y', X', \PP'\times \PP, \QQ'\times \QQ)})} 
\ar[drr] _-{\sim} 
&& 
{v ^{!} \circ pr ^{!} _1  (\FF _{(Y', X', \PP', \QQ')})} 
\ar[ll] ^-{v ^{!} \circ pr ^{!} _1 \circ \psi ^\sharp _{pr _1}} _-{\sim}
\ar[d] ^-{\psi _{pr _1}} _-{\sim}
\\
{pr ^{\prime !} _1 \circ pr '_{1+}  \circ v ^{!}  (\FF _{(Y', X', \PP'\times \PP, \QQ'\times \QQ)})} 
\ar[rr]  _-{\sim}
\ar[d] ^-{\sim} _-{\psi _v} 
&& 
{v ^{!}   (\FF _{(Y', X', \PP'\times \PP, \QQ'\times \QQ)})} 
\ar[d] ^-{\sim} _-{\psi _v} 
\\
{pr ^{\prime !} _1 \circ pr '_{1+}   (\FF _{(Y'', X'', \PP''\times \PP, \QQ''\times \QQ)})} 
&
{pr ^{\prime !} _1  (\FF _{(Y'', X'', \PP'', \QQ'')})} 
\ar[r]  ^-{\psi _{pr '_1}} _-{\sim}
\ar[l]  ^-{\psi ^\sharp _{pr '_1}} _-{\sim}
& 
{v ^{!}   (\FF _{(Y', X', \PP'\times \PP, \QQ'\times \QQ)})} 
}
\end{equation}
est l'image par $pr ^{\prime !} _1$ du carré de droite de \ref{eqcat-YXvsP-diag}.
Comme le diagramme \ref{eqcat-YXvsP-diag-dcom} est commutatif, on en déduit le résultat.

\end{proof}

\section{Modules surcohérents et isocristaux surcohérents}

\begin{defi}
\label{d-uni}
Soit $(Y,X)/K$ un couple.
On note $d\text{-}\mathfrak{Uni} (Y,X/K)$ 
la catégorie des $d$-cadres localement propres au-dessus de $(Y,X)$.
Un objet de $d\text{-}\mathfrak{Uni} (Y,X/K)$ est ainsi la donnée d'un $d$-cadre localement propre 
$(\QQ', \PP', T', X', Y')$  
et d'un morphisme de 
$\mathfrak{Cpl}$
de la forme $(b,a) \colon (Y',X') \to (Y,X)$. 
On note $(b,a) \colon (\QQ', \PP', T', X', Y') \to (Y,X)$ un tel objet ou par abus de notations $(\QQ', \PP', T', X', Y')$.
Les morphismes 
$((\QQ'', \PP'', T'', X'', Y''), (b',a') )
\to
((\QQ', \PP', T', X', Y'), (b,a) )$
de $d\text{-}\mathfrak{Uni} (Y,X/K)$ sont les morphismes 
$(f,g,c,d)\colon (\QQ'', \PP'', T'', X'', Y'')
\to
(\QQ', \PP', T', X', Y')$
de cadres localement propres
tels que 
$(b,a) \circ (d,c) =(b',a') $.
On peut abusivement noter
$u\colon (\QQ'', \PP'', T'', X'', Y'')
\to
(\QQ', \PP', T', X', Y')$ un tel morphisme. 
\end{defi}

\begin{nota}
\label{isoc-def}
Soit $(\PP, T,  X, Y)$ 
un $d$-cadre tel que $Y$ soit lisse.
On note 
$F\text{-}\mathrm{Isoc} ^{\dag \dag}  (\PP, T, X, Y)$
la sous-catégorie pleine de
$F\text{-}\mathrm{Surhol}  (\PP, T, X, Y)$ des $F$-isocristaux surconvergents
sur $ (\PP, T, X, Y)$, i.e. 
des objets dans l'image essentielle du foncteur 
$\sp _{X \to \PP, T,+}$.
\end{nota}

\begin{lemm}
\label{lemm-annul-div}
Soit 
$u= (f, a,b) \colon (\PP', T', X', Y') \to (\PP, T,  X, Y)$ 
un morphisme de $d$-cadres.
\begin{enumerate}
\item Si $Y$ est lisse, le foncteur 
$u ^{!} [-d _{Y'/Y}]$ est exact sur 
$F\text{-}\mathrm{Isoc} ^{\dag \dag}  (\PP, T, X, Y)$.
\item Si $b$ est lisse, alors le foncteur 
$u ^{!} [-d _{Y'/Y}]$ est exact sur 
$F\text{-}\mathrm{Surhol}  (\PP, T, X, Y)$.
\end{enumerate}

\end{lemm}

\begin{proof}
Comme $T$ et $T'$ sont des diviseurs, il suffit de le vérifier en dehors de $T'$ (via l'argument bien connu de  \cite[4.3.12]{Be1}). 
De plus, comme cela est local, on peut supposer $Y'$ lisse. 
Le lemme est alors immédiat grâce au théorème de Berthelot-Kashiwara.
\end{proof}

\begin{rema}
\label{rema-cadvsd-cad}
Le lemme \ref{lemm-annul-div} et faux pour les morphismes de cadres . 
Par exemple, si $\PP= \widehat{\A} ^2 _{\V}$ et $U$ est $P$ privé de l'origine noté $O$ et 
$u\colon (\PP,  O, P, U) \to (\PP, \emptyset, P,P)$ est le morphisme canonique.
\end{rema}

\begin{defi}
\label{defi-isoc}
Grâce au lemme \ref{lemm-annul-div}, 
dans la définition de \ref{defi-surcoh-cadre},
en remplaçant les catégories de la forme 
$F\text{-}D ^\mathrm{b} _\mathrm{surcoh}$ par celles de la forme $F\text{-}\mathrm{Isoc} ^{\dag \dag}$,
en remplaçant $\mathfrak{Uni} (Y,X/K)$
par $d\text{-}\mathfrak{Uni} (Y,X/K)$
et  
en remplaçant les termes de la forme
$u ^! $ par $u ^![-d _{Y''/Y'}]$, on construit
la catégorie $F\text{-}\mathrm{Isoc} ^{\dag \dag} (Y,X/K)$ des $F$-isocristaux surcohérents sur $(Y,X)/K$.
\end{defi}

\begin{defi}
\label{d-uni-zar}
Soit $(Y,X)/K$ un couple.
On note $d\text{-}\mathfrak{Zar} (Y,X/K)$ 
la sous-catégorie pleine de 
$d\text{-}\mathfrak{Uni} (Y,X/K)$ 
dont les objets 
$(\QQ', \PP', T', X', Y')$ sont tels que 
la flèche structurale $Y' \to Y$ soit une immersion ouverte.

Grâce au lemme \ref{lemm-annul-div}, 
dans la définition de \ref{defi-surcoh-cadre},
en remplaçant les catégories de la forme 
$F\text{-}D ^\mathrm{b} _\mathrm{surcoh}$ par celles de la forme $F\text{-}\mathrm{Surcoh}$,
en remplaçant $\mathfrak{Uni} (Y,X/K)$
par $d\text{-}\mathfrak{Zar} (Y,X/K)$, 
on construit
la catégorie $F\text{-}\mathrm{Surcoh}(Y,X/K)$ des $\D$-modules arithmétiques surcohérents sur $(Y,X)/K$.
 
\end{defi}

\begin{vide}
Soit $(\QQ,\PP, T,X, Y)$  un $d$-cadre localement propre.
De manière analogue à \ref{eqcat-YXvsP}, 
les foncteurs canoniques de restriction 
$F\text{-}\mathrm{Isoc} ^{\dag \dag} (Y,X/K)
\to 
F\text{-}\mathrm{Isoc} ^{\dag \dag}  (\QQ,\PP, T,X, Y)$
et
$F\text{-}\mathrm{Surcoh}(Y,X/K)
\to 
F\text{-}\mathrm{Surcoh}  (\QQ,\PP, T,X, Y)$
sont des équivalences de catégories.
Pour tout couple $(Y,X)/K$, 
on vérifie de plus 
que $F\text{-}\mathrm{Isoc} ^{\dag \dag} (Y,X/K)$ est une sous-catégorie pleine de 
$F\text{-}\mathrm{Surcoh}(Y,X/K)$.

\end{vide}

\section{Les trois opérations cohomologiques pour les catégories de complexes de type surcohérent}

\begin{vide}
\label{prod-tens-surcoh}
Soit $(Y,X)/K$ un couple.
On définit le bifoncteur  produit tensoriel
 $$-
\smash{\overset{\L}{\otimes}}   ^{\dag}
_{\O  _{(Y, X)/K}}
-
\colon
F\text{-}D ^\mathrm{b} _\mathrm{surcoh} (\D ^\dag _{(Y,X)/K})
\times 
F\text{-}D ^\mathrm{b} _\mathrm{surcoh} (\D ^\dag _{(Y,X)/K})
\to 
F\text{-}D ^\mathrm{b} _\mathrm{surcoh} (\D ^\dag _{(Y,X)/K})$$ 
en posant, pour tous objets  
$\E= (\E _{(Y', X', \PP', \QQ')} , \phi _u ) $ 
et
$\FF= (\FF _{(Y', X', \PP', \QQ')} , \psi _u ) $ 
de 
$F\text{-}D ^\mathrm{b} _\mathrm{surcoh} (\D ^\dag _{(Y,X)/K})$ 
$$\E
\smash{\overset{\L}{\otimes}}   ^{\dag}
_{\O  _{(Y, X)/K}}
\FF
:= (\E _{(Y', X', \PP', \QQ')}
\smash{\overset{\L}{\otimes}}   ^{\dag}
_{\O  _{(Y', \PP')}}
\FF _{(Y', X', \PP', \QQ')}
[ -d _{Y'/Y}], 
\theta _u),$$
où pour tout morphisme
$u\colon  (Y'', X'', \PP'', \QQ'') \to(Y', X', \PP', \QQ')$ de $\mathfrak{Uni} (Y,X/K)$, 
$\theta _u$ est l'isomorphisme composé
\begin{gather}
\notag
\theta _u
\colon 
u ^{!} (\E _{(Y', X', \PP', \QQ')}
\smash{\overset{\L}{\otimes}}   ^{\dag}
_{\O  _{(Y', \PP')}}
\FF _{(Y', X', \PP', \QQ')}
[ -d _{Y'/Y}])
\riso 
u ^{!} (\E _{(Y', X', \PP', \QQ')} )
\smash{\overset{\L}{\otimes}}   ^{\dag}
_{\O  _{(Y'', \PP'')}}
u ^{!} (\FF _{(Y', X', \PP', \QQ')})
[ -d _{Y''/Y}]
\\
\notag
\underset{\phi _u \otimes \psi _u}{\riso} 
\E _{(Y'', X'', \PP'', \QQ'')}
\smash{\overset{\L}{\otimes}}   ^{\dag}
_{\O  _{(Y'', \PP'')}}
\FF _{(Y'', X'', \PP'', \QQ'')}
[ -d _{Y''/Y}].
\end{gather}
\end{vide}

\begin{vide}
\label{iminvextsurcoh}
Soit $(b,a) \colon (Y', X' ) \to (Y,X)$ un morphisme de couples. 
Comme $\mathfrak{Uni} (Y',X'/K)$ est une sous-catégorie 
de
$\mathfrak{Uni} (Y,X/K)$, on dispose donc du foncteur restriction
$F\text{-}D ^\mathrm{b} _\mathrm{surcoh} (\D ^\dag _{(Y,X)/K})
\to 
F\text{-}D ^\mathrm{b} _\mathrm{surcoh} (\D ^\dag _{(Y',X')/K})$
que l'on notera $(b,a) ^{!}$ et que l'on appellera aussi image inverse extraordinaire par 
$(b,a)$.
Ce foncteur est clairement transitif: 
pour tout autre morphisme de couples de la forme 
$(b',a') \colon (Y'', X'' ) \to (Y',X')$, on bénéficie de l'isomorphisme canonique:
\begin{equation}
\label{iminvextsurcoh-trans}
(b',a') ^{!} \circ (b,a) ^{!} 
\riso 
(b\circ b',a\circ a') ^{!}.
\end{equation}
 
\end{vide}

\begin{defi}
\label{defi-realis}
Soit $(b,a) \colon (\widetilde{Y},\widetilde{X}) \to (Y,X)$ un morphisme complet de couples. 
\begin{itemize}
\item On dit que $(b,a)$ est strictement réalisable si
pour tout objet $(Y', X', \PP', \QQ')$ de $\mathfrak{Uni} (Y,X/K)$
il existe 
un objet de $\mathfrak{Uni} (\widetilde{Y},\widetilde{X}/K)$
de la forme 
$(\widetilde{Y} \times _Y Y',\widetilde{X} \times _X X', \widetilde{\PP} ',\widetilde{\QQ} ')$.

\item On dit que $(b,a)$ est réalisable si
pour tout objet $(Y', X', \PP', \QQ')$ de $\mathfrak{Uni} (Y,X/K)$
il existe 
un morphisme de couples 
de la forme
$(id, c) \colon 
(\widetilde{Y} \times _Y Y',\widetilde{X} ')
\to 
(\widetilde{Y} \times _Y Y',\widetilde{X} \times _X X')$
avec $c$ propre
et
un objet de 
$\mathfrak{Cad}$
de la forme 
$(\widetilde{Y} \times _Y Y',\widetilde{X}', \widetilde{\PP} ',\widetilde{\QQ} ')$.

Remarquons que quitte à considérer 
$\widetilde{\QQ} ' \times \QQ$
à la place de 
$\widetilde{\QQ} '$, on peut supposer qu'il existe
un morphisme de la forme 
$(b',a'\circ c,g',f')\colon 
(\widetilde{Y} \times _Y Y',\widetilde{X} ', \widetilde{\PP} ',\widetilde{\QQ} ')
\to 
(Y', X', \PP', \QQ')$
avec $f'$ lisse, 
$\widetilde{\PP} '= f ^{\prime -1} (\PP ')$ et où $b'$ et $a'$ sont les projections canoniques
(en effet, comme $a$ et $c$ sont propres, alors $a' \circ c$ aussi ; par conséquent le plongement de
$ \widetilde{X} '$ dans 
$ \widetilde{\PP} '$ est fermé). 

\end{itemize}
\end{defi}

\begin{prop}
\label{theo-iminvextsurcoh}
Soit $(b,a) \colon (Y', X' ) \to (Y,X)$ un morphisme complet de couples.
\begin{enumerate}
\item Si $a$ est projectif alors $(b,a)$ est strictement réalisable. 
\item  Si $b$ est quasi-projectif, alors  $(b,a)$ est réalisable. 
\end{enumerate}
\end{prop}

\begin{proof}
Lorsque $a$ est projectif, il est immédiat que $(b,a)$ soit strictement réalisable.
Lorsque $b$ est quasi-projectif, 
on reprend l'argument technique invoqué lors de la preuve du théorème analogue en cohomologie rigide (i.e. \cite[2.3.5]{Berig}):
grâce au lemme de Chow précis de Gruson-Rayaud (voir \cite[5.7.14]{Gruson_Raynaud-platprojectif}): 
il existe un morphisme de couples
$(id,c)\colon (Y', X'') \to (Y', X ') $ tel que $c$ soit projectif et $a \circ c$ soit projectif, ce qui entraîne que $(a,b)$ est réalisable. 
\end{proof}

\begin{vide}
[Construction de l'image directe par un morphisme réalisable de couples]
\label{defi-ba+}
Soit $(b,a) \colon (\widetilde{Y},\widetilde{X}) \to (Y,X)$ un morphisme réalisable de couples  (voir
\ref{defi-realis}). 
On définit le foncteur image directe par $(b,a)$ de la manière suivante. 
Soit 
$\widetilde{\E}= (\widetilde{\E} _{(\widetilde{Y} ',\widetilde{X}', \widetilde{\PP} ',\widetilde{\QQ} ')} , \widetilde{\phi} _u ) \in 
F\text{-}D ^\mathrm{b} _\mathrm{surcoh} (\D ^\dag _{(\widetilde{Y} ,\widetilde{X})/K})$.
Pour tout objet $(Y', X', \PP', \QQ')$ de $\mathfrak{Uni} (Y,X/K)$, 
choisissons 
un morphisme de couples 
de la forme
$(id, c) \colon 
(\widetilde{Y} \times _Y Y',\widetilde{X} ')
\to 
(\widetilde{Y} \times _Y Y',\widetilde{X} \times _X X')$
avec $c$ propre
tel qu'il existe
un morphisme de $\mathfrak{Uni} (Y,X/K)$ de la forme 
$(b, a) _{(Y', X', \PP', \QQ')}:= (b',a'\circ c,g',f')\colon 
(\widetilde{Y} \times _Y Y',\widetilde{X} ', \widetilde{\PP} ',\widetilde{\QQ} ')
\to 
(Y', X', \PP', \QQ')$
avec $f'$ lisse et où $b'$ et $a'$ sont les projections canoniques 
(il en existe au moins un d'après \ref{defi-realis}). 
On pose 
$$\E _{(Y', X', \PP', \QQ')}:=
(b, a) _{(Y', X', \PP', \QQ')+}  
 (\widetilde{\E} _{(\widetilde{Y} \times _Y Y',\widetilde{X} ', \widetilde{\PP} ',\widetilde{\QQ} ')} ).$$
D'après \ref{4.2.3.4}, cet objet ne dépend pas, à isomorphisme canonique près, 
du choix d'un tel morphisme  $(b, a) _{(Y', X', \PP', \QQ')}$.
Pour toute flèche
$u\colon  (Y'', X'', \PP'', \QQ'') \to(Y', X', \PP', \QQ')$ de $\mathfrak{Uni} (Y,X/K)$,
avec \ref{4.2.3.4}, 
quitte à changer
$(b, a) _{(Y'', X'', \PP'', \QQ'')}$, 
on peut choisir
$(b, a) _{(Y'', X'', \PP'', \QQ'')}:= (b'',a'',g'',f'')\colon 
(\widetilde{Y} \times _Y Y'',\widetilde{X} '  \times _{X'} X'',  \widetilde{\PP} '\times _{\PP '} \PP'' ,\widetilde{\QQ} '\times _{\QQ '} \QQ'' )
\to 
(Y'', X'', \PP'', \QQ'')$ égal à la projection canonique. 
Notons alors $\widetilde{u}\colon 
(\widetilde{Y} \times _Y Y'',\widetilde{X} '  \times _{X'} X'',  \widetilde{\PP} '\times _{\PP '} \PP'' ,\widetilde{\QQ} '\times _{\QQ '} \QQ'' )
\to 
(\widetilde{Y} \times _Y Y',\widetilde{X} ', \widetilde{\PP} ',\widetilde{\QQ} ')$ la projection canonique. 
On note $\phi _u$ l'isomorphisme composé:
\begin{gather}
\notag
\phi _u \colon
u ^{!} (\E _{(Y', X', \PP', \QQ')}) =
u ^{!}\circ (b, a) _{(Y', X', \PP', \QQ')+}  
 (\widetilde{\E} _{(\widetilde{Y} \times _Y Y',\widetilde{X} ', \widetilde{\PP} ',\widetilde{\QQ} ')} )
\underset{\ref{iso-chgt-base-cadre}}{\riso}
(b, a) _{(Y'', X'', \PP'', \QQ'')+} \circ 
\widetilde{u} ^!  (\widetilde{\E} _{(\widetilde{Y} \times _Y Y',\widetilde{X}', \widetilde{\PP} ',\widetilde{\QQ} ')} )
\\
\underset{\widetilde{\phi} _{\widetilde{u}}}{\riso}
(b, a) _{(Y'', X'', \PP'', \QQ'')+} (\widetilde{\E} _{(\widetilde{Y} \times _Y Y'',\widetilde{X} '  \times _{X'} X'',  \widetilde{\PP} '\times _{\PP '} \PP'' ,\widetilde{\QQ} '\times _{\QQ '} \QQ'' )} )
=\E _{(Y'', X'', \PP'', \QQ'')}.
\end{gather}
Le foncteur image directe par $(b,a)$, noté
$(b, a) _+ \colon F\text{-}D ^\mathrm{b} _\mathrm{surcoh} (\D ^\dag _{(\widetilde{Y} ,\widetilde{X})/K})
\to F\text{-}D ^\mathrm{b} _\mathrm{surcoh} (\D ^\dag _{(Y,X)/K})$, 
est défini en posant 
$$(b, a) _+  (\widetilde{\E} _{(\widetilde{Y} ',\widetilde{X}', \widetilde{\PP} ',\widetilde{\QQ} ')} , \widetilde{\phi} _u ) :=
(u ^{!} (\E _{(Y', X', \PP', \QQ')}), \phi _u).
$$
\end{vide}

\begin{vide}
Soit $(b,a) \colon (Y', X' ) \to (Y,X)$ 
et
$(b',a') \colon (Y'', X'' ) \to (Y',X')$
deux morphismes (resp. strictement) réalisables de couples.
Alors $(b\circ b',a\circ a')\colon \colon (Y'', X'' ) \to (Y,X)$
est un morphisme (resp. strictement) réalisable de couples. 
De plus, 
on bénéficie de l'isomorphisme canonique de transitivité:
\begin{equation}
\label{imdirsurcoh-trans}
(b,a) _+ \circ (b',a') _+
\riso 
(b\circ b',a\circ a') _+ .
\end{equation}

\end{vide}

\begin{prop}
\label{ind-theo-iminvextsurcoh}
Soit $(b,a) \colon (Y', X' ) \to (Y,X)$ un morphisme de couples avec $a$ propre et $b$ un isomorphisme. 
Alors 
les foncteurs $(b,a) ^{!} $ et $(b,a) _{+} $
induisent des équivalences canoniquement quasi-inverses entre 
$F\text{-}D ^\mathrm{b} _\mathrm{surcoh} (\D ^\dag _{(Y,X)/K})$
et
$F\text{-}D ^\mathrm{b} _\mathrm{surcoh} (\D ^\dag _{(Y',X')/K})$.
\end{prop}

\begin{proof}
Grâce au lemme de Chow précis de Gruson-Rayaud (voir \cite[5.7.14]{Gruson_Raynaud-platprojectif}),
il existe un morphisme de couples
$(id,c)\colon (Y', X'') \to (Y', X ') $ tel que $c$ soit projectif et $a \circ c$ soit projectif.
Par construction du foncteur $(b, a) _+$ décrite dans \ref{defi-ba+}, on obtient alors
$(b, a) _+ =(b, a \circ c)  _{+} \circ   (id, c) ^{!}$.
Quitte à décomposer $(b,a)$ en 
$(id, a) \circ (b,id)$, comme le cas où $a=id$ est évident, on se ramène à supposer $b=id$.
Il résulte alors du lemme \ref{4.2.3.4} que les foncteurs 
$(id,c) _+$ et $(id,c) ^{!}$ 
(resp. $(id, a \circ c)  _{+}$ et $(id, a \circ c)  ^{!}$) 
induisent alors des équivalences quasi-inverses 
entre les catégories 
$F\text{-}D ^\mathrm{b} _\mathrm{surcoh} (\D ^\dag _{(Y,X'')/K})$
et
$F\text{-}D ^\mathrm{b} _\mathrm{surcoh} (\D ^\dag _{(Y,X')/K})$
(resp. 
$F\text{-}D ^\mathrm{b} _\mathrm{surcoh} (\D ^\dag _{(Y,X'')/K})$
et
$F\text{-}D ^\mathrm{b} _\mathrm{surcoh} (\D ^\dag _{(Y,X)/K})$).
D'où le résultat.
\end{proof}

\section{Indépendance par rapport à la compactification partielle}

\begin{defi}
\label{def-Sc^-1}
Notons $S _c \subset \mathrm{Fl}(\mathfrak{Cpl})$ (voir les notations de \ref{defi-couples})
le système multiplicatif à gauche des morphismes de la forme $(id,a) \colon (Y, X' ) \to (Y,X)$ avec $a$ propre. 
D'après \cite[7.1.16]{Kashiwara-schapira-book},
on dispose alors de la catégorie localisée $S _c  ^{-1} \mathfrak{Cpl}$.
Un morphisme $u \colon (Y', X' ) \to (Y,X)$ de $S _c  ^{-1} \mathfrak{Cpl}$
est dit {\og complet\fg} s'il existe un représentant de $u$ de la forme
$ (Y', X' ) \underset{(id, c)}{\longleftarrow} (Y',Z' ) \underset{(b,a)}{\longrightarrow} (Y,X)$ 
avec $a$ propre. 
\end{defi}

\begin{rema}
\label{rel-eq-S-1Hom}
\begin{itemize}
\item Soient $(Y', X' ) $ et $(Y, X) $ deux objets de $\mathfrak{Cpl}$.
On remarque que si $u$ et $v $ sont deux morphismes de $\mathfrak{Cpl}$ tels que 
$u, v\circ u \in S _c$ alors $v \in S _c$. 
On en déduit que 
 $\mathrm{Hom} _{S _c  ^{-1} \mathfrak{Cpl}} ( (Y', X' ) ,~(Y, X) )$
est la famille des couples de flèches de la forme 
$ (Y', X' ) \underset{(id, c)}{\longleftarrow} (Y',Z') \underset{(b,a)}{\longrightarrow} (Y,X)$ 
avec $c$ un morphisme propre
quotienté par la relation d'équivalence suivante: 
deux telles paires de morphismes
$ (Y', X' ) \underset{(id, c)}{\longleftarrow} (Y',Z' _1) \underset{(b,a)}{\longrightarrow} (Y,X)$ 
et
$ (Y', X' ) \underset{(id, c')}{\longleftarrow} (Y',Z' _2) \underset{(b',a')}{\longrightarrow} (Y,X)$ 
sont équivalentes s'il existe deux morphismes de $S _c$ de la forme
$(id, c _1)\colon (Y',Z') \to (Y',Z' _1)$
et
$(id, c _2)\colon (Y',Z') \to (Y',Z' _2)$
tels que 
$(id, c ) \circ (id, c _1)= (id, c ') \circ (id, c _2)$
et
$(b, a ) \circ (id, c _1)= (b', a ') \circ (id, c _2)$.

\item Pour vérifier que cette relation est bien une relation d'équivalence, 
on peut se passer 
de l'axiome S4' de \cite[7.1.5]{Kashiwara-schapira-book}
en la remplaçant par 
la propriété que si $(id, c _1)\colon (Y',Z' _1) \to (Y',Z')$
et
$(id, c _2)\colon (Y',Z'_2) \to (Y',Z' )$ sont deux morphismes de 
$S _c$,
alors on dispose du produit fibré $(Y',Z' _1 \times _Z Z' _2) $ et des projections
$(id, p _1) \colon (Y',Z' _1 \times _Z Z' _2) \to (Y',Z' _1)$ et $(id, p _2) \colon  (Y',Z' _1 \times _Z Z' _2) \to (Y',Z' _2)$ 
qui sont aussi des morphismes de $S _c$.

\item La composition d'un morphisme
de $S _c  ^{-1} \mathfrak{Cpl}$ 
représenté par 
$(Y', X' ) \underset{(id, c)}{\longleftarrow} (Y',Z') \underset{(b,a)}{\longrightarrow} (Y,X)$
avec un second représenté par 
$ (Y'', X'' ) \underset{(id, c')}{\longleftarrow} (Y'',Z'') \underset{(b',a')}{\longrightarrow} (Y',X')$
est le morphisme 
représenté par 
$ (Y'', X'' ) \underset{(id, c'\circ p )}{\longleftarrow} (Y'',Z'' \times _{X'} Z' ) \underset{(b \circ b',a\circ q)}{\longrightarrow} (Y,X)$,
où 
$p \colon Z'' \times _{X'} Z' \to Z''$
et 
$q \colon Z'' \times _{X'} Z' \to Z'$ sont les morphismes canoniques.
De même, pour vérifier que ceci est défini, on peut éviter d'avoir recours 
à l'axiome S4' de \cite[7.1.5]{Kashiwara-schapira-book}. 

\end{itemize}

\end{rema}

\begin{vide}
\label{vide-4.6}
\begin{enumerate}
\item  Soit $u \colon (Y', X' ) \to (Y,X)$ un morphisme de $S _c  ^{-1} \mathfrak{Cpl}$.
On remarque que grâce au lemme de Chow précis de Gruson-Rayaud (voir \cite[5.7.14]{Gruson_Raynaud-platprojectif}), 
on peut supposer qu'un représentant de $u$ soit de la forme 
$ (Y', X' ) \underset{(id, c)}{\longleftarrow} (Y',Z') \underset{(b,a)}{\longrightarrow} (Y,X)$ 
 avec $c$ un morphisme projectif. 
De plus, si $u \colon (Y', X' ) \to (Y,X)$ est un isomorphisme de $S _c ^{-1} \mathfrak{Cpl}$, 
on peut choisir un tel représentant 
avec $a$ et $c$ projectifs, $b$ un isomorphisme de $k$-variétés.

\item On dispose dans $S _c  ^{-1} \mathfrak{Cpl}$ de produits fibrés.
Plus précisément, soient $u \colon (Y', X' ) \to (Y,X)$ et 
$v \colon (Y'', X'' ) \to (Y,X)$ deux morphismes de $S _c  ^{-1} \mathfrak{Cpl}$.
Choisissons 
$ (Y', X' ) \underset{(id, c)}{\longleftarrow} (Y',Z') \underset{(b,a)}{\longrightarrow} (Y,X)$ 
et
$ (Y'', X'' ) \underset{(id, c')}{\longleftarrow} (Y'',Z'') \underset{(b',a')}{\longrightarrow} (Y,X)$ 
des représentants de $u$ et $v$.
Alors le produit fibré 
$(Y'', X'' ) \times _{(Y, X )} (Y', X' )$ dans $S _c  ^{-1} \mathfrak{Cpl}$
est égal, à isomorphisme de $S _c  ^{-1} \mathfrak{Cpl}$ canonique près,
à $(Y''\times _{Y} Y', Z''\times _{X} Z ' )$.

\end{enumerate}

\end{vide}

\begin{defi}
 \label{uni-sharp}

\begin{itemize}
\item 
On note $T _c \subset \mathrm{Fl}(\mathfrak{Cad}) $  (voir \ref{defi-couples})
le système multiplicatif à gauche des flèches de la forme $(id,a, g,f) \colon (Y, X', \PP', \QQ' ) \to (Y,X, \PP, \QQ)$ avec $a$ propre 
(pour valider l'axiome S4' de \cite[7.1.5]{Kashiwara-schapira-book}, on utilise 
\cite[5.7.14]{Gruson_Raynaud-platprojectif} et \ref{theo-iminvextsurcoh}).
On dispose alors de 
la catégorie $ T _c ^{-1}\mathfrak{Cad}$.  
Le foncteur canonique de restriction
$\mathfrak{Cad} \to \mathfrak{Cpl} $ défini par 
$ (Y,X, \PP, \QQ)\mapsto  (Y,X)$ et $(b,a, g,f)\mapsto (b,a)$ 
se factorise en le foncteur 
$\mathrm{r}\colon 
T _c ^{-1}\mathfrak{Cad} \to S  _c ^{-1}\mathfrak{Cpl} $.
 
\item Soit $(Y,X)/K$ un couple.
On note $U _c ^{-1}\mathfrak{Uni} (Y,X/K) $ 
la catégorie dont les objets sont les pairs 
$((Y', X', \PP', \QQ'), ~u)$ où 
$(Y', X', \PP', \QQ')$ est 
un cadre localement propre 
et où $u \colon  (Y', X') \to (Y, X)$ est un morphisme de $S _c ^{-1} \mathfrak{Cpl}$. 
On pourra noter un tel objet par abus de notations 
$u\colon (Y', X', \PP', \QQ') \to  (Y, X)$
 ou encore
$(Y', X', \PP', \QQ') $.
Un morphisme $(Y'', X'', \PP'', \QQ''), ~u') \to (Y', X', \PP', \QQ'), ~u)$ 
de $U _c ^{-1}\mathfrak{Uni} (Y,X/K) $ 
est un morphisme de $ T _c ^{-1}\mathfrak{Cad}$ de la forme 
$\phi \colon (Y'', X'', \PP'', \QQ'')\to (Y', X', \PP', \QQ')$ au-dessus de $(Y,X)$,
i.e. tel que $u \circ \mathrm{r} (\phi) = u '$. 
\end{itemize}
\end{defi}

\begin{vide}
\label{im-inv-S-1}
Soit $u \colon (Y', X', \PP', \QQ' ) \to (Y,X, \PP, \QQ)$ un morphisme de $ T _c ^{-1}\mathfrak{Cad}$.
On définit canoniquement le fonceur $u ^{!}$ de la manière suivante. 
Si $(Y', X', \PP', \QQ' ) \underset{s}{\longleftarrow} (Y',X'', \PP'', \QQ'') \underset{v}{\longrightarrow} (Y,X, \PP, \QQ)$
est un représentant de $u$, on pose
$u ^{!} := s _{+} \circ v ^{!}$. 
Grâce à \ref{ind-theo-iminvextsurcoh},
on vérifie que cela ne dépend pas, à isomorphisme canonique près du choix du représentant de $u$.
De plus,
si $u '\colon (Y',X'', \PP'', \QQ'')  \to (Y', X', \PP', \QQ' )$ est un morphisme de $ T _c ^{-1}\mathfrak{Cad}$, 
on dipose de l'isomorphisme canonique 
$u ^{\prime !} \circ u ^{!}\riso (u \circ u ') ^{!} $.
\end{vide}

\begin{defi}
\label{defi-surcoh-cadre-indX}
Soit $(Y,X)/K$ un couple.
En remplaçant respectivement dans la définition
\ref{defi-surcoh-cadre}
les catégories $ \mathfrak{Cpl}$ par $S _c ^{-1} \mathfrak{Cpl}$
et $\mathfrak{Uni} (Y,X/K) $ 
par 
$U _c ^{-1}\mathfrak{Uni} (Y,X/K) $,
grâce aussi à \ref{im-inv-S-1},
on définit la catégorie $F\text{-}\widetilde{D} ^\mathrm{b} _\mathrm{surcoh} (\D ^\dag _{(Y,X)/K})$ 
des complexes de type surcohérent sur $(Y,X)/K$.
\end{defi}

\begin{prop}
\label{prop-surcoh-indX}
Soit $(Y,X)/K$ un couple.
Le foncteur de restriction
$F\text{-}\widetilde{D} ^\mathrm{b} _\mathrm{surcoh} (\D ^\dag _{(Y,X)/K})
\to 
F\text{-}D ^\mathrm{b} _\mathrm{surcoh} (\D ^\dag _{(Y,X)/K})$
est une équivalence de catégories.
\end{prop}

\begin{proof}
Construisons canoniquement un foncteur quasi-inverse.
Soit $(\E _{(Y', X', \PP', \QQ')} , \phi _u )$ un objet de la catégorie
$F\text{-}D ^\mathrm{b} _\mathrm{surcoh} (\D ^\dag _{(Y,X)/K})$.
Soit $((Y',  \widetilde{X}', \widetilde{\PP}', \widetilde{\QQ}') , u)$ un objet de $U _c ^{-1}\mathfrak{Uni} (Y,X/K) $.
Choisissons un représentant de $u$ de la forme
$(Y', \widetilde{X}') \underset{(id,c)}{\longleftarrow} (Y',  X') 
\underset{u}{\longrightarrow}
 (Y,X)$ avec $c$ projectif. 
 D'après \ref{theo-iminvextsurcoh}.1,
 il existe un morphisme de $\mathfrak{Cad}$
 de la forme
 $(id, c, g,f)\colon 
(Y',  \widetilde{X}', \widetilde{\PP}', \widetilde{\QQ}')
\to 
(Y', X', \PP', \QQ')$.
On pose alors 
$\E _{(Y',  \widetilde{X}', \widetilde{\PP}', \widetilde{\QQ}')}:= 
(id, c, g,f) _{+} (\E _{(Y', X', \PP', \QQ'))}$.
 On vérifie alors que 
l'on obtient bien un objet de 
$F\text{-}\widetilde{D} ^\mathrm{b} _\mathrm{surcoh} (\D ^\dag _{(Y,X)/K})$.
\end{proof}

\begin{defi}
\label{def-real-S-1}
Soit $u \colon (\widetilde{Y},\widetilde{X}) \to (Y,X)$ un morphisme complet de $S _c ^{-1} \mathfrak{Cpl}$.
On dit que $u$ est {\og réalisable\fg}  si
pour tout objet $(Y', X', \PP', \QQ')$ de $U _c ^{-1}\mathfrak{Uni} (Y,X/K)$
il existe alors un isomorphisme de
$S  _c ^{-1} \mathfrak{Cpl}$
de la forme
$
(\widetilde{Y} ',\widetilde{X} ')
\riso 
(\widetilde{Y} \times _Y Y',\widetilde{X} \times _X X')$
et
un objet de 
$\mathfrak{Cad}$
de la forme 
$(\widetilde{Y} ',\widetilde{X}', \widetilde{\PP} ',\widetilde{\QQ} ')$.

\end{defi}

\begin{lemm}
Soit
$u \colon (Y', X' ) \to (Y,X)$
un morphisme de $ \mathfrak{Cpl}$.
Le morphisme $u$ est 
réalisable comme morphisme de $ \mathfrak{Cpl}$
si et seulement si 
$u$ est réalisable comme morphisme de $S _c ^{-1} \mathfrak{Cpl}$.
\end{lemm}

\begin{proof}
Cela résulte aussitôt de \ref{theo-iminvextsurcoh}.1, 
\ref{vide-4.6}.1 et de \ref{vide-4.6}.2.
\end{proof}

\begin{defi}
Soit $u \colon (Y', X' ) \to (Y,X)$ un morphisme de $S _c  ^{-1} \mathfrak{Cpl}$.

\begin{itemize}
\item On dispose du foncteur de restriction
$u ^{!}\colon 
F\text{-}\widetilde{D} ^\mathrm{b} _\mathrm{surcoh} (\D ^\dag _{(Y,X)/K})
\to
F\text{-}\widetilde{D} ^\mathrm{b} _\mathrm{surcoh} (\D ^\dag _{(Y',X')/K})$
dit image inverse extraordinaire par $u$.

\item Si $u$ est réalisable, 
on définit le foncteur 
$u _{+}\colon F\text{-}\widetilde{D} ^\mathrm{b} _\mathrm{surcoh} (\D ^\dag _{(Y',X')/K})
\to F\text{-}\widetilde{D} ^\mathrm{b} _\mathrm{surcoh} (\D ^\dag _{(Y,X)/K})$
image directe par $u$ 
de manière identique à \ref{defi-ba+}.
\end{itemize}

De plus, on dispose du produit tensoriel sur 
$ F\text{-}\widetilde{D} ^\mathrm{b} _\mathrm{surcoh} (\D ^\dag _{(Y,X)/K})$.

\end{defi}

\section{Le foncteur dual, catégorie de type dual surcohérent}

\begin{defi}
\label{def-surcoh-couple*}
Soit $(Y,X)/K$ un couple.
On définit la catégorie 
$F\text{-}D ^\mathrm{b} _\mathrm{surcoh} (\D ^\dag _{(Y,X)/K}) ^{*}$ 
des complexes de type dual surcohérent sur $(Y,X)/K$
de la manière suivante:
\begin{itemize}
\item Un objet est la donnée 
\begin{itemize}
\item  d'une famille d'objets $\E _{(Y', X', \PP', \QQ')} $ de 
$F\text{-}D ^\mathrm{b} _\mathrm{surcoh}  (Y', X', \PP',\QQ'/K)$, 
où $(Y', X', \PP', \QQ')$ parcourt les objets de $\mathfrak{Uni} (Y,X/K)$ ;
\item pour toute flèche
$u\colon  (Y'', X'', \PP'', \QQ'') \to(Y', X', \PP', \QQ')$ de $\mathfrak{Uni} (Y,X/K)$,
d'un isomorphisme 
$$\phi _u \colon u ^{+} (\E _{(Y', X', \PP', \QQ')} ) \riso \E _{(Y'', X'', \PP'', \QQ'')} $$ 
dans $F\text{-}D ^\mathrm{b} _\mathrm{surcoh}  (Y'', X'', \PP'',\QQ''/K)$,
ces isomorphismes vérifiant la condition de cocycle : pour tous morphismes 
$u\colon  (Y'', X'', \PP'', \QQ'') \to(Y', X', \PP', \QQ')$ et
$v\colon  (Y''', X''', \PP''', \QQ''') \to(Y'', X'', \PP'', \QQ'')$
de $\mathfrak{Uni} (Y,X/K)$,
le diagramme
\begin{equation}
\xymatrix @R=0,3cm{
 {v ^{+} \circ u ^{+} (\E _{(Y', X', \PP', \QQ')} ) } 
 \ar[r] ^-{v ^{+} (\phi _u)}
 \ar[d] ^-{\sim}
 & 
 {v ^{+} ( \E _{(Y'', X'', \PP'', \QQ'')} )} 
  \ar[d] ^-{\phi _v}
 \\ 
 {(u\circ v) ^{+} (\E _{(Y', X', \PP', \QQ')} ) } 
  \ar[r] ^-{\phi _{u\circ v}}
 & 
 {\E _{(Y''', X''', \PP''', \QQ''')} } 
 }
\end{equation}

soit commutatif.
\end{itemize}
On notera de manière elliptique 
$(\E _{(Y', X', \PP', \QQ')} , \phi _u )$ une telle donnée.

\item Un morphisme $\alpha\colon (\E _{(Y', X', \PP', \QQ')} , \phi _u ) 
\to 
(\FF _{(Y', X', \PP', \QQ')} , \psi _u )$
de
$F\text{-}D ^\mathrm{b} _\mathrm{surcoh} (\D ^\dag _{(Y,X)/K})$ 
est la donnée d'une famille 
de morphismes
$\alpha _{(Y', X', \PP', \QQ')}\colon \E _{(Y', X', \PP', \QQ')} \to \FF _{(Y', X', \PP', \QQ')} $
 de 
$F\text{-}D ^\mathrm{b} _\mathrm{surcoh}  (Y', X', \PP',\QQ'/K)$
telle que pour tout morphisme
$u\colon  (Y'', X'', \PP'', \QQ'') \to(Y', X', \PP', \QQ')$ 
de 
$\mathfrak{Uni} (Y,X/K)$ on ait 
$\psi _u \circ u ^{+}(\alpha _{(Y', X', \PP', \QQ')})
=
\alpha _{(Y'', X'', \PP'', \QQ'')} \circ \phi _u $.
\end{itemize}

\end{defi}

\begin{vide}
\label{dual-surcoh*}
Soit $(Y,X)/K$ un couple.
On dispose du foncteur 
$\DD _{(Y,X)/K}
\colon 
F\text{-}D ^\mathrm{b} _\mathrm{surcoh} (\D ^\dag _{(Y,X)/K})
\to 
F\text{-}D ^\mathrm{b} _\mathrm{surcoh} (\D ^\dag _{(Y,X)/K}) ^*$
défini par 
$(\E _{(Y', X', \PP', \QQ')} , \phi _u )\mapsto
(\DD _{Y', \PP'}(\E _{(Y', X', \PP', \QQ')}) , \psi _u )$, où 
$\psi _u$ est l'isomorphisme composé:
\begin{equation}
\psi _u \colon u ^{+} (\DD _{Y', \PP'} (\E _{(Y', X', \PP', \QQ')} ))
\riso
\DD _{Y'', \PP''} \circ u ^{!} (\E _{(Y', X', \PP', \QQ')} )
\underset{\DD _{Y'', \PP''} (\phi _u)}{\liso}
\DD _{Y'', \PP''} ( \E _{(Y'', X'', \PP'', \QQ'')}  ).
\end{equation}
De même, on dispose 
du foncteur 
$\DD _{(Y,X)/K}
\colon 
F\text{-}D ^\mathrm{b} _\mathrm{surcoh} (\D ^\dag _{(Y,X)/K}) ^*
\to 
F\text{-}D ^\mathrm{b} _\mathrm{surcoh} (\D ^\dag _{(Y,X)/K})$
et de l'isomorphisme canonique de bidualité 
$\DD _{(Y,X)/K} \circ \DD _{(Y,X)/K} \riso Id$.
\end{vide}

\begin{vide}
Soit $(Y,X)/K$ un couple.
On définit par dualité via \ref{dual-surcoh*}
les trois opérations cohomologiques duales sur les catégories de type duale surcohérent 
de celles sur les catégories de type surcohérente.

$\bullet$ Conformément aux notations de T. Abe (voir \cite[5.8]{Abe-Frob-Poincare-dual}), 
on notera le bifoncteur  produit tensoriel tordu (on pourrait aussi dire dualisé) de la manière suivante:
 $$-
\smash{\overset{\L}{\widetilde{\otimes}}}   ^{\dag}
_{\O  _{(Y, X)/K}}
-
\colon 
F\text{-}D ^\mathrm{b} _\mathrm{surcoh} (\D ^\dag _{(Y,X)/K}) ^{*}
\times 
F\text{-}D ^\mathrm{b} _\mathrm{surcoh} (\D ^\dag _{(Y,X)/K}) ^{*}
\to 
F\text{-}D ^\mathrm{b} _\mathrm{surcoh} (\D ^\dag _{(Y,X)/K})^{*}, $$ 
ce dernier étant défini en posant, 
pour tout $\E, \FF \in F\text{-}D ^\mathrm{b} _\mathrm{surcoh} (\D ^\dag _{(Y,X)/K}) ^{*}$, 
$$\E
\smash{\overset{\L}{\widetilde{\otimes}}}   ^{\dag}
_{\O  _{(Y, X)/K}}
\FF
:= 
\DD _{(Y,X)/K}  \left (\DD _{(Y,X)/K}  (\E)
\smash{\overset{\L}{\otimes}}   ^{\dag}
_{\O  _{(Y, X)/K}}
\DD _{(Y,X)/K}  (\FF)
\right ).$$

$\bullet$ 
Soit $u \colon (\widetilde{Y} ,\widetilde{X}) \to (Y,X)$ un morphisme réalisable de $ \mathfrak{Cpl}$.
Le foncteur image directe extraordinaire par $u$, noté
$u _! \colon F\text{-}D ^\mathrm{b} _\mathrm{surcoh} (\D ^\dag _{(\widetilde{Y} ,\widetilde{X})/K}) ^{*}
\to F\text{-}D ^\mathrm{b} _\mathrm{surcoh} (\D ^\dag _{(Y,X)/K})^{*}$, 
est défini en posant, pour tout $\widetilde{\E} \in F\text{-}D ^\mathrm{b} _\mathrm{surcoh} (\D ^\dag _{(\widetilde{Y} ,\widetilde{X})/K}) ^{*}$, 
$$u _!  (\widetilde{\E}) := 
\DD _{(Y,X)/K}\circ u _+ \circ \DD _{(\widetilde{Y} ,\widetilde{X})/K} (\widetilde{\E}).$$

$\bullet$ 
Soit $u \colon (Y', X' ) \to (Y,X)$ un morphisme de $ \mathfrak{Cpl}$.
On définit le foncteur image inverse en posant:
$u ^{+} \riso \DD _{(Y',X')/K}\circ u ^{!}  \circ \DD _{(Y,X)/K} $.
\end{vide}

\begin{vide}
De manière analogue à \ref{defi-surcoh-cadre-indX}, on définit
la catégorie 
$F\text{-}\widetilde{D} ^\mathrm{b} _\mathrm{surcoh} (\D ^\dag _{(Y,X)/K}) ^*$.
Toutes les propriétés de ce chapitre s'étendent naturellement à cette situation, 
i.e. il suffit de remplacer partout {\og $F\text{-}D ^\mathrm{b} _\mathrm{surcoh}$\fg} par {\og $F\text{-}\widetilde{D} ^\mathrm{b} _\mathrm{surcoh}$\fg}
et de rajouter des $S ^{-1} _c$.
\end{vide}

\section{Le formalisme des six opérations sur les catégories de complexes de type surholonome sur les couples}

\begin{defi}
\label{defi-surhol-couple}
Soit $(Y,X)/K$ un couple.
On définit la catégorie $F\text{-}D ^\mathrm{b} _\mathrm{surhol} (\D ^\dag _{(Y,X)/K})$ 
des complexes de type surholonome sur $(Y,X)/K$
de la manière suivante:
un objet est la donnée $(\E, \FF, \epsilon )$ où $\E \in F\text{-}D ^\mathrm{b} _\mathrm{surcoh} (\D ^\dag _{(Y,X)/K})$, 
$\FF \in F\text{-}D ^\mathrm{b} _\mathrm{surcoh} (\D ^\dag _{(Y,X)/K}) ^*$
et $\epsilon \colon 
\DD _{(Y,X)/K} (\E) \riso \FF$ 
est un isomorphisme de $F\text{-}D ^\mathrm{b} _\mathrm{surcoh} (\D ^\dag _{(Y,X)/K}) ^*$.
Un morphisme
$(\E', \FF', \epsilon ' )\to (\E, \FF, \epsilon )$ est la donnée 
des morphismes $f \colon \E' \to \E$ de 
$F\text{-}D ^\mathrm{b} _\mathrm{surcoh} (\D ^\dag _{(Y,X)/K}) $
et $g \colon \FF \to \FF'$ de $F\text{-}D ^\mathrm{b} _\mathrm{surcoh} (\D ^\dag _{(Y,X)/K}) ^*$
tels que $g \circ \epsilon  = \epsilon ' \circ \DD _{(Y,X)/K} (f)$.
\end{defi}

\begin{vide}
Soit $(Y,X)/K$ un couple.
Les foncteurs 
$F\text{-}D ^\mathrm{b} _\mathrm{surcoh} (\D ^\dag _{(Y,X)/K})
\to 
F\text{-}D ^\mathrm{b} _\mathrm{surhol} (\D ^\dag _{(Y,X)/K})$
et
$F\text{-}D ^\mathrm{b} _\mathrm{surhol} (\D ^\dag _{(Y,X)/K}) 
\to 
F\text{-}D ^\mathrm{b} _\mathrm{surcoh} (\D ^\dag _{(Y,X)/K})$
définis respectivement par 
$\E \mapsto 
(\E, \DD _{(Y,X)/K} (\E)  , id)$
et $(\E, \FF, \epsilon ) \mapsto \E$ sont des équivalences quasi-inverses de catégories. 
Cependant, contrairement a priori à $F\text{-}D ^\mathrm{b} _\mathrm{surcoh} (\D ^\dag _{(Y,X)/K})$,
on peut définir canoniquement le foncteur dual 
$$\DD _{(Y,X)/K}\colon 
F\text{-}D ^\mathrm{b} _\mathrm{surhol} (\D ^\dag _{(Y,X)/K})
\to 
F\text{-}D ^\mathrm{b} _\mathrm{surhol} (\D ^\dag _{(Y,X)/K})$$ 
en posant 
$\DD _{(Y,X)/K}(\E, \FF, \epsilon ) := (\DD _{(Y,X)/K} (\FF) , \DD _{(Y,X)/K} (\E) , \epsilon ')$, 
où  
$\epsilon '$ est l'isomorphisme composé: 
$$\epsilon ' \colon \DD _{(Y,X)/K} \circ \DD _{(Y,X)/K} (\FF)
\riso 
\FF 
\underset{\epsilon ^{-1}}{\riso}
 \DD _{(Y,X)/K} (\E) .$$

\end{vide}

\begin{vide}
Soit $(Y,X)/K$ un couple.
On définit le bifoncteur produit tensoriel
 $$-
\smash{\overset{\L}{\otimes}}   ^{\dag}
_{\O  _{(Y, X)/K}}
-
\colon
F\text{-}D ^\mathrm{b} _\mathrm{surhol} (\D ^\dag _{(Y,X)/K})
\times 
F\text{-}D ^\mathrm{b} _\mathrm{surhol} (\D ^\dag _{(Y,X)/K})
\to 
F\text{-}D ^\mathrm{b} _\mathrm{surhol} (\D ^\dag _{(Y,X)/K})$$ 
en posant, pour tous objets  
$(\E, \FF, \epsilon )$, 
$(\E', \FF', \epsilon ')$ de
$F\text{-}D ^\mathrm{b} _\mathrm{surhol} (\D ^\dag _{(Y,X)/K})$:
$$(\E, \FF, \epsilon )
\smash{\overset{\L}{\otimes}}   ^{\dag}
_{\O  _{(Y, X)/K}}
(\E', \FF', \epsilon '):=
(\E
\smash{\overset{\L}{\otimes}}   ^{\dag}
_{\O  _{(Y, X)/K}}
\E', 
\FF
\smash{\overset{\L}{\widetilde{\otimes}}}   ^{\dag}
_{\O  _{(Y, X)/K}}
\FF', \epsilon ''),$$
où $\epsilon''$ est canoniquement isomorphe à 
$\DD _{(Y,X)/K}( \DD _{(Y,X)/K} (\epsilon )
\smash{\overset{\L}{\otimes}}   ^{\dag}
_{\O  _{(Y, X)/K}}
\DD _{(Y,X)/K} (\epsilon ')) $.

\end{vide}

\begin{vide}
Soit $u \colon (Y', X' ) \to (Y,X)$ un morphisme de $\mathfrak{Cpl}$.
\begin{itemize}
\item On dispose du foncteur image inverse extraordinaire par $u$
$$u ^{!}\colon F\text{-}D ^\mathrm{b} _\mathrm{surhol} (\D ^\dag _{(Y,X)/K})
\to 
F\text{-}D ^\mathrm{b} _\mathrm{surhol} (\D ^\dag _{(Y',X')/K})$$
défini en posant, pour tout 
$(\E, \FF, \epsilon )$ de
$F\text{-}D ^\mathrm{b} _\mathrm{surhol} (\D ^\dag _{(Y,X)/K})$,
$$u ^{!} (\E, \FF, \epsilon )
:=
(u ^{!} (\E), u ^{+} (\FF), \epsilon '),$$ 
où 
$\epsilon '$ est l'isomorphisme canoniquement isomorphe à 
$u ^{+} (\epsilon)$.

\item Le foncteur image inverse par $u$
$u ^{+}\colon F\text{-}D ^\mathrm{b} _\mathrm{surhol} (\D ^\dag _{(Y,X)/K})
\to 
F\text{-}D ^\mathrm{b} _\mathrm{surhol} (\D ^\dag _{(Y',X')/K})$
est défini en posant 
$u ^{+}:= \DD _{(Y',X')/K} \circ u ^{!}\circ \DD _{(Y,X)/K}$.

\end{itemize}
\end{vide}

\begin{vide}
Soit $u \colon (Y', X' ) \to (Y,X)$ un morphisme réalisable de $ \mathfrak{Cpl}$.

\begin{itemize}
\item Le foncteur image directe par $u$
$$u _{+}\colon F\text{-}D ^\mathrm{b} _\mathrm{surhol} (\D ^\dag _{(Y',X')/K})
\to 
F\text{-}D ^\mathrm{b} _\mathrm{surhol} (\D ^\dag _{(Y,X)/K})$$
est défini en posant, 
pour tout objet
$(\E', \FF', \epsilon ')$ de
$F\text{-}D ^\mathrm{b} _\mathrm{surhol} (\D ^\dag _{(Y,X)/K})$,
$$u _{+}(\E', \FF', \epsilon ')
:=
(u _{+}(\E'), u _{!}( \FF'), \epsilon ),$$ 
où 
$\epsilon $ est l'isomorphisme canoniquement isomorphe à 
$u _{!} (\epsilon ')$.

\item Le foncteur image directe extraordinaire par $u$
$$u _{!}\colon F\text{-}D ^\mathrm{b} _\mathrm{surhol} (\D ^\dag _{(Y',X')/K})
\to 
F\text{-}D ^\mathrm{b} _\mathrm{surhol} (\D ^\dag _{(Y,X)/K})$$
est défini en posant 
$u _{!}:= \DD _{(Y,X)/K} \circ u _{+}\circ \DD _{(Y',X')/K}$.
\end{itemize}

\end{vide}

\section{Le formalisme des six opérations sur les catégories de complexes de type surholonome sur les variétés}

\begin{defi}
\label{uni-var}
On note $\mathfrak{Var}$ la catégorie des $k$-variétés.
Soit $Y$ un objet de $\mathfrak{Var}$.
On définit $\mathfrak{Uni} (Y/K)$, la catégorie des cadres propres  au-dessus de $Y/K$, 
de manière analogue à 
$\mathfrak{Uni} (Y,X/K) $  en remplaçant 
la catégorie $ \mathfrak{Cpl}$ par 
$\mathfrak{Var}$
et $\mathfrak{Cad} $ 
par 
$\mathfrak{Cadp} $. 
Un objet de $\mathfrak{Uni} (Y/K)$ est ainsi la donnée d'un cadre propre  $(Y',  \PP')$  et d'un morphisme de 
$k$-variétés 
de la forme $b \colon Y' \to Y$. 
On notera $b\colon (Y', \PP') \to Y$ un tel objet ou plus simplement par abus de notations $(Y', \PP') $.
Les morphismes 
$((Y'', \PP''), b')
\to
((Y',  \PP'), b)$
de $\mathfrak{Uni} (Y/K)$ sont les morphismes de cadres propres
$(c, f)\colon (Y'', \PP'')
\to
(Y', \PP')$
tels que 
$b \circ c =b'$.
On pourra noter abusivement 
$u\colon (Y'', \PP'')
\to
(Y', \PP')$ un tel morphisme. 
\end{defi}

\begin{defi}
\label{defi-Uni-var}
Soit $Y$ une variété sur $k$.
En remplaçant respectivement dans la définition
\ref{defi-surcoh-cadre}
les catégories $ \mathfrak{Cpl}$ par 
$\mathfrak{Var}$
et $\mathfrak{Uni} (Y,X/K) $ 
par 
$\mathfrak{Uni} (Y/K) $,
on définit la catégorie $F\text{-}D ^\mathrm{b} _\mathrm{surcoh} (\D ^\dag _{Y/K})$ 
des complexes de type surcohérent sur $Y/K$.

\end{defi}

\begin{vide}
Soit  $(Y, \PP)$ un cadre propre.
De manière analogue à \ref{eqcat-YXvsP},
on vérifie que le foncteur canonique de restriction 
$F\text{-}D ^\mathrm{b} _\mathrm{surcoh} (\D ^\dag _{Y/K})
\to 
F\text{-}D ^\mathrm{b} _\mathrm{surcoh}  (Y, \PP/K)$
est une équivalence de catégories.
\end{vide}

\begin{defi}
 \label{uni-sharp-var}

\begin{itemize}
\item 
On note $V _c \subset \mathrm{Fl}(\mathfrak{Cadp}) $ 
le système multiplicatif à gauche des flèches de la forme $(id,g) \colon (Y, \PP') \to (Y, \PP)$ 
(pour valider l'axiome S4' de \cite[7.1.5]{Kashiwara-schapira-book}, on utilise 
\cite[5.7.14]{Gruson_Raynaud-platprojectif} et \ref{theo-iminvextsurcoh}).
On dispose alors de 
la catégorie $ V _c ^{-1}\mathfrak{Cadp}$.  
Le foncteur canonique de restriction
$\mathfrak{Cadp} \to \mathfrak{Var} $ défini par 
$ (Y,\PP)\mapsto Y$ et $(b ,g)\mapsto b$ 
se factorise en le foncteur 
$\mathrm{s}\colon 
V _c ^{-1}\mathfrak{Cadp} \to \mathfrak{Var} $.
 
\item Soit $Y$ une variété. 
On note $W _c ^{-1}\mathfrak{Uni} (Y/K) $ 
la catégorie dont les objets sont ceux de $\mathfrak{Uni} (Y/K) $
et 
dont les morphismes de la forme $(Y'',  \PP''), ~u') \to (Y', \PP'), ~u)$ 
sont les morphismes  
$\phi \colon (Y'',  \PP'')\to (Y', \PP')$ 
de $ V _c ^{-1}\mathfrak{Cadp} $ au-dessus de $Y$,
i.e. tel que $u \circ \mathrm{s} (\phi) = u '$. 

\item 
Comme pour \ref{defi-surcoh-cadre-indX}, 
en remplaçant respectivement dans la définition
\ref{defi-Uni-var}
la catégorie 
et $\mathfrak{Uni} (Y/K) $ 
par 
$W _c ^{-1}\mathfrak{Uni} (Y/K) $,
grâce aussi à \ref{im-inv-S-1},
on définit la catégorie $F\text{-}\widetilde{D} ^\mathrm{b} _\mathrm{surcoh} (\D ^\dag _{Y/K})$ 
des complexes de type surcohérent sur $(Y,X)/K$.
De même que pour 
\ref{prop-surcoh-indX},
on vérifie que le foncteur de restriction
$F\text{-}\widetilde{D} ^\mathrm{b} _\mathrm{surcoh} (\D ^\dag _{Y/K})
\to 
F\text{-}D ^\mathrm{b} _\mathrm{surcoh} (\D ^\dag _{Y/K})$
est une équivalence de catégories.

\end{itemize}
\end{defi}

\begin{vide}
\begin{itemize}
\item Soit $b \colon Y' \to Y$ un morphisme de $k$-variétés. 
Comme $\mathfrak{Uni} (Y'/K)$ est une sous-catégorie 
de
$\mathfrak{Uni} (Y/K)$, on dispose donc du foncteur restriction
$F\text{-}D ^\mathrm{b} _\mathrm{surcoh} (\D ^\dag _{Y/K})
\to 
F\text{-}D ^\mathrm{b} _\mathrm{surcoh} (\D ^\dag _{Y'/K})$
que l'on notera $b ^{!}$ et que l'on appellera aussi image inverse extraordinaire par 
$b$.

\item De manière analogue à \ref{prod-tens-surcoh}, on définit le bifoncteur produit tensoriel
que l'on notera
$$-
\smash{\overset{\L}{\otimes}}   ^{\dag}
_{\O  _{Y/K}}
-
\colon
F\text{-}D ^\mathrm{b} _\mathrm{surcoh} (\D ^\dag _{Y/K})
\times 
F\text{-}D ^\mathrm{b} _\mathrm{surcoh} (\D ^\dag _{Y/K})
\to 
F\text{-}D ^\mathrm{b} _\mathrm{surcoh} (\D ^\dag _{Y/K}).$$ 
\end{itemize}

\end{vide}

\begin{defi}
\label{defi-abs-real}
Soit $b \colon \widetilde{Y} \to Y$ un morphisme de $k$-variétés. 
On dira que $b$ est réalisable si 
pour tout objet $(Y',\PP')$ de $\mathfrak{Uni} (Y/K)$,
il existe un objet de 
$\mathfrak{Cadp} $
de la forme 
$(\widetilde{Y} \times _Y Y', \widetilde{\PP} ')$.
Remarquons que, 
quitte à considérer 
$ \widetilde{\PP} ' \times \PP$ 
à la place de 
$ \widetilde{\PP} '$, on peut supposer qu'il existe
un morphisme de la forme 
$(b',f')\colon 
(\widetilde{Y} \times _Y Y', \widetilde{\PP} ')
\to 
(Y',\PP')$
avec $f'$ lisse et où $b'$ est la projection canonique. 
Si  $b \colon \widetilde{Y} \to Y$ un morphisme réalisable de $k$-variétés, 
alors, de manière analogue à \ref{defi-ba+}, 
on définit le foncteur image directe que l'on notera 
$$ b _{+}\colon 
F\text{-}D ^\mathrm{b} _\mathrm{surcoh} (\D ^\dag _{\widetilde{Y}/K})
\to F\text{-}D ^\mathrm{b} _\mathrm{surcoh} (\D ^\dag _{Y/K}).$$
\end{defi}

\begin{rema}
Soit $u=(b,a)\colon (Y',X')\to (Y,X)$ un morphisme de couples 
avec $X$ et $X'$ propres. Alors $u$ est réalisable si et seulement si $b$ est réalisable. 
En effet, cela résulte facilement de 
\ref{theo-iminvextsurcoh}.1 et \cite[5.7.14]{Gruson_Raynaud-platprojectif}.
\end{rema}

\begin{defi}
Soit $Y$ une variété sur $k$.
Comme pour \ref{def-surcoh-couple*}, en remplaçant les images inverses extraordinaires par 
les images inverses, on définit la catégorie
$F\text{-}D ^\mathrm{b} _\mathrm{surcoh} (\D ^\dag _{Y/K}) ^{*}$
des complexes de type dual surcohérent sur $Y$.
Le foncteur dual induit l'équivalence canonique de catégories 
$\DD _{Y/K} \colon 
F\text{-}D ^\mathrm{b} _\mathrm{surcoh} (\D ^\dag _{Y/K})
\cong
F\text{-}D ^\mathrm{b} _\mathrm{surcoh} (\D ^\dag _{Y/K}) ^{*}$.
On définit alors, comme pour \ref{defi-surhol-couple}, 
la catégorie $F\text{-}D ^\mathrm{b} _\mathrm{surhol} (\D ^\dag _{Y/K})$
des complexes de type surholonome sur $Y$. 
\end{defi}

\begin{vide}
Comme pour le chapitre précédent, 
on définit les six opérations cohomogiques sur les catégories de complexes de type surholonomes sur les $k$-variétés.
Pour toute $k$-variété $Y$, 
on note 
$\DD _{Y/K} \colon 
F\text{-}D ^\mathrm{b} _\mathrm{surhol} (\D ^\dag _{Y/K})
\to 
F\text{-}D ^\mathrm{b} _\mathrm{surhol} (\D ^\dag _{Y/K})$
le foncteur dual
et
 $-
\smash{\overset{\L}{\otimes}}   ^{\dag}
_{\O  _{Y/K}}
-$
le bifoncteur produit tensoriel.
Pour tout morphisme $b \colon \widetilde{Y}\to Y$ de $k$-variétés, 
on note $b ^{!}$ l'image inverse extraordinaire et 
$b ^{+}$ l'image inverse. 
Pour tout morphisme réalisable $b \colon \widetilde{Y}\to Y$ de $k$-variétés,
on note $b _{!}$ l'image directe extraordinaire et 
$b _{+}$ l'image directe.

\end{vide}

\begin{vide}
\label{defi-surhol-Y}
$\bullet$ Soient $(Y',X')/K$ et $(Y,X)/K$
deux couples tels que $X$ et $X'$ soient propres.
L'application canonique de restriction
$\mathrm{Hom} _{S _c  ^{-1} \mathfrak{Cpl}} ( (Y', X' ) ,~(Y, X) )
\to
\mathrm{Hom} _{\mathfrak{Var}} ( Y' ,~Y )$
est alors une bijection.
En effet,
soit $b\colon Y' \to Y$ un morphisme. 
Alors, $Y'$ est une sous-variété de
$X ' \times X$ via le graphe de $b$. 
En notant $Z'$ l'adhérence de $Y'$ dans
$X ' \times X$, on obtient alors le morphisme de $S _c  ^{-1} \mathfrak{Cpl}$ représenté par
$ (Y', X' ) \underset{(id, c)}{\longleftarrow} (Y',Z' ) \underset{(b,a)}{\longrightarrow} (Y,X)$.
D'où la surjectivité. 
Soient deux morphismes de $S _c  ^{-1} \mathfrak{Cpl}$ représentés respectivement par
$ (Y', X' ) \underset{(id, c)}{\longleftarrow} (Y',Z' ) \underset{(b,a)}{\longrightarrow} (Y,X)$
et
$ (Y', X' ) \underset{(id, c')}{\longleftarrow} (Y',Z'' ) \underset{(b,a')}{\longrightarrow} (Y,X)$.
Posons $T := Z ' \times _{X'} Z''$, 
$p \colon T \to Z'$ et $q \colon T \to Z''$ les projections canoniques,
$f := a \circ p$ et $g := a ' \circ q$ et 
$T' := \ker (f,g)$ et $i\colon T ' \to T$ l'immersion fermée canonique (on rappelle que 
$\ker (f,g)= T \times _{T \times X} T$, les deux morphismes $T \to T \times X$ étant donnés respectivement par le graphe de $f$ et de $g$). 
On vérifie alors 
que $ (Y', X' ) \underset{(id, c\circ p\circ i)}{\longleftarrow} (Y',T' ) \underset{(b,a\circ p\circ i)}{\longrightarrow} (Y,X)$
est équivalent à ces deux morphismes qui sont par conséquent égaux.

$\bullet$ Soient $(Y', X', \PP', \QQ' )$ et $ (Y,X, \PP, \QQ)$  deux objets de
$ T _c ^{-1}\mathfrak{Cad}$ tels que $X'$ et $X$ soient propres.  
On a de même 
$\mathrm{Hom} _{T _c ^{-1}\mathfrak{Cad}} ( (Y', X', \PP', \QQ' ),~(Y,X, \PP, \QQ))
= 
\mathrm{Hom} _{V ^{-1} _c\mathfrak{Cadp}} 
( (Y', \PP'),~(Y,\PP))
=
\mathrm{Hom} _{\mathfrak{Var}} ( Y' ,~Y )$.

$\bullet$ On déduit des deux premiers points que
les catégories  
$F\text{-}\widetilde{D} ^\mathrm{b} _\mathrm{surcoh} (\D ^\dag _{Y/K})$ 
et 
$F\text{-}\widetilde{D} ^\mathrm{b} _\mathrm{surcoh} (\D ^\dag _{(Y,X)/K})$ sont canoniquement égales.
Comme les foncteurs canoniques de restriction
$F\text{-}\widetilde{D} ^\mathrm{b} _\mathrm{surcoh} (\D ^\dag _{(Y,X)/K})
\to 
F\text{-}D ^\mathrm{b} _\mathrm{surcoh} (\D ^\dag _{(Y,X)/K})$
et
$F\text{-}\widetilde{D} ^\mathrm{b} _\mathrm{surcoh} (\D ^\dag _{Y/K})
\to 
F\text{-}D ^\mathrm{b} _\mathrm{surcoh} (\D ^\dag _{Y/K})$
sont des équivalences de catégories, 
on en déduit qu'il en est de même du foncteur restriction
$$F\text{-}D ^\mathrm{b} _\mathrm{surcoh} (\D ^\dag _{Y/K})
\to 
F\text{-}D ^\mathrm{b} _\mathrm{surcoh} (\D ^\dag _{(Y,X)/K}).$$
De même, le foncteur restriction
$F\text{-}D ^\mathrm{b} _\mathrm{surhol} (\D ^\dag _{Y/K})
\to 
F\text{-}D ^\mathrm{b} _\mathrm{surhol} (\D ^\dag _{(Y,X)/K})$
est une équivalence de catégories.
\end{vide}

\bibliographystyle{smfalpha}
\providecommand{\bysame}{\leavevmode ---\ }
\providecommand{\og}{``}
\providecommand{\fg}{''}
\providecommand{\smfandname}{et}
\providecommand{\smfedsname}{\'eds.}
\providecommand{\smfedname}{\'ed.}
\providecommand{\smfmastersthesisname}{M\'emoire}
\providecommand{\smfphdthesisname}{Th\`ese}

\bigskip
\noindent Daniel Caro\\
Laboratoire de Mathématiques Nicolas Oresme\\
Université de Caen
Campus 2\\
14032 Caen Cedex\\
France.\\
email: daniel.caro@unicaen.fr

\end{document}